\documentclass[a4paper,12pt,leqno]{article}
\usepackage{amsmath}
\usepackage{amsthm}
\usepackage{amssymb}
\usepackage{amscd}
\usepackage{graphicx, color}
\usepackage{a4wide}
\usepackage[all]{xy}
\newtheorem{theorem}{Theorem}[section]
\newtheorem{lemma}[theorem]{Lemma}
\newtheorem{corollary}[theorem]{Corollary}
\newtheorem{proposition}[theorem]{Proposition}

\theoremstyle{definition}
\newtheorem{definition}[theorem]{Definition}
\newtheorem{example}[theorem]{Example}

\theoremstyle{remark}
\newtheorem{remark}[theorem]{Remark}

\numberwithin{equation}{section}
\theoremstyle{remark}
\usepackage[all]{xy}

\newcommand{\Hol}{\mbox{{\rm Hol}}}

\newcommand{\Z}{\Bbb Z}
\newcommand{\C}{\Bbb C}
\newcommand{\R}{\Bbb R}

\newcommand{\T}{\Bbb T}
\renewcommand{\P}{{\rm P}}
\newcommand{\SP}{\mbox{{\rm SP}}}

\newcommand{\Map}{\mbox{{\rm Map}}}

\newcommand{\CP}{\Bbb C {\rm P}}

\newcommand{\dis}{\displaystyle}
\newcommand{\p}{\prime}

\newcommand{\E}{E(d)}
\newcommand{\Ed}{E(d+1)}
\newcommand{\SZ}{{\mathcal{X}}^{d}}
\newcommand{\SZd}{{\mathcal{X}}^{d+1}}

\newcommand{\I}{\mbox{{\rm (i)}}}
\newcommand{\II}{\mbox{{\rm (ii)}}}
\newcommand{\III}{\mbox{{\rm (iii)}}}
\newcommand{\IV}{\mbox{{\rm (iv)}}}

\newcommand{\XS}{X_{\bf \Sigma}}

\newcommand{\rmin}{r_{\rm min}}
\title{\bf The homotopy type of spaces of coprime polynomials revisited}

\author{Andrzej Kozlowski\footnote{%
Institute of Applied Mathematics and Mechanics,
University of Warsaw, Banacha 2, 02-097 Warsaw, Poland
(E-mail: akoz@mimuw.edu.pl)
}
\  and \ 
Kohhei Yamaguchi\footnote{%
Department of Mathematics,
University of Electro-Communications,  Chofu, Tokyo 182-8585, Japan
(E-mail: kohhe@im.uec.ac.jp);
The second author was supported by 
JSPS KAKENHI Grant Number 23540079, 26400083.
\newline
\quad 2010 {\it Mathematics Subject Clasification.} Primary 55P10; Secondly 55R80, 55P35, 14M25.}}

\date{}

\begin{document}
\maketitle

\begin{abstract}
The purpose of this paper is to
study the topology of certain toric varieties $X_I$, arising as quotients of the action of $\C^*$ on complements of arrangements of coordinate subspaces in $\C^n$, and
to improve the homotopy
stability dimension for the inclusion map
$\Hol_d^*(S^2,X_I)\to \Map_d^*(S^2,X_I)$
given in \cite{GKY1} by using
spectral sequences induced from simplicial resolutions. 
\end{abstract}

\section{Introduction.}\label{section 1}

\paragraph{\bf 1.1. Coordinate subspaces and the spaces $X_I$.}
\ Let $n\geq 2$ be a positive integer and let
$[n]$ denote the set $[n]=\{0,1,2,\cdots ,n-1\}$.
For each subset $\sigma
=\{i_1,\cdots ,i_s\}\subset [n]$, let
$L_{\sigma}\subset \C^n$ denote
{\it the coordinate subspace in }$\C^n$  defined by
\begin{equation}\label{Lsigma}
L_{\sigma}
=\{(x_0,x_1,\cdots ,x_{n-1})\in\C^n:x_{i_1} =\cdots =x_{i_s}=0\}.
\end{equation}
Let $I$ be any collection of subsets of $[n]$
such that $\mbox{card}(\sigma)\geq 2$ for all $\sigma \in I$, 
where $\mbox{card}(\sigma)$ denotes the number of elements in $\sigma$.
Let $Y_I\subset \C^n$ be {\it the complement  of the arrangement of coordinate subspaces}
defined by
\begin{equation}\label{YI}
Y_I=\C^n\setminus\bigcup_{\sigma\in I}L_{\sigma}
=\C^n\setminus L(I),
\quad \mbox{where we set }\  L(I)=\bigcup_{\sigma\in I}L_{\sigma}.
\end{equation}
Consider the natural free $\C^*$-action on $Y_I$ given by 
the coordinate-wise multiplication
and let $X_I$ denote the orbit space given by
\begin{equation}\label{equ: XI}
X_I=Y_I/\C^*
=(\C^n\setminus L(I))/\C^*.
\end{equation}
Note that $X_I$ coincides with the complex variety considered in \cite[page 437]{GKY1}, and 
that there is a principal $\C^*$-bundle
\begin{equation}\label{equ: pI}
Y_I \stackrel{p_I}{\longrightarrow} X_I.
\end{equation}
\begin{example}\label{exa: XI}
{\rm
(i) 
If $I=I(n)=\{\{0,1,\cdots ,n-1\}\}$,  
$L(I(n))=
\{{\bf 0}\}$
and 
we can identify $X_{I(n)}$ with the $(n-1)$-dimensional complex
projective space $\CP^{n-1}$, i.e.
$X_{I(n)}=(\C^n\setminus \{{\bf 0}\})/\C^*=\CP^{n-1}.$
\par
(ii)
If $n\geq 3$ and $I=J(n)=\{\{i,j\}:0\leq i<j<n\}$, 
we   can be identify $X_{J(n)}$ with the subspace of
$\CP^{n-1}$
given by $X_{J(n)}=\CP^{n-1}\setminus\bigcup_{0\leq i<j<n}H_{i,j},$
where
$H_{i,j}=\{[x_0:\cdots :x_{n-1}]\in\CP^{n-1}:x_i=x_j=0\}.$%
\footnote{%
 To simplify the notation we will write 
$X_n$ for $X_{J(n)}$, as in \cite{GKY1}.
}
\par
\par
(iii)
In general, we  easily see that
$X_I=\CP^{n-1}\setminus \bigcup_{\sigma\in I}H_{\sigma},$
where
\begin{equation*}\label{exa: XI}
H_{\sigma}=\{[x_0:\cdots :x_{n-1}]\in\CP^{n-1}:x_j=0
\mbox{ for all }j\in\sigma\}.
\qed
\end{equation*}
}
\end{example}

The algebraic torus $\T^{n-1}_{\C}=(\C^*)^{n-1}$ acts on $X_I$
in the natural manner
\begin{equation}\label{equ: T-action}
(t_1,\cdots ,t_{n-1})\cdot
[x_0:\cdots x_{n-1}]=[x_0:t_1x_1:\cdots :t_{n-1}x_{n-1}]
\end{equation}
for 
$((t_1,\cdots ,t_{n-1}),[x_0:\cdots :x_{n-1}])\in
\T^{n-1}_{\C}\times X_I$,
and it is easy to see that $X_I$ is a  smooth toric variety.
Note that  $X_I$ is a non-compact toric variety (its fan is not complete) if $I\not= I(n).$
\paragraph{\bf 1.2.  The simplicial complex $K(I)$.}
\ There is an alternative and better known way to construct  the spaces $X_I$.
Let $[n]=\{0,1,2,\cdots ,n-1\}$ and
recall that \textit{a simplicial complex} $K$ on an index set $[n]$
is a collection of  subsets $\sigma$ of $[n]$ which satisfies the condition that
 any $\tau \subset \sigma$ is
contained in $K$ if $\sigma\in K$.%
\footnote{In this paper a simplicial complex means {\it an abstract} 
simplicial complex and we assume that any simplicial complex contains 
the empty set
$\emptyset$.}
For a simplicial complex $K$ on the index set $[n]$, let
$U(K)$ denote \textit{the complement of the arrangement of coordinate subspaces}
 given by
\begin{equation}
U(K)=\C^n\setminus \bigcup_{\sigma\notin K,\sigma\subset [n]}L_{\sigma}.
\end{equation}
Now recall the following useful result.
\begin{lemma}[\cite{BP}, Prop. 8.6]\label{lmm: KI}
Let
$K(I)$ denote the set of subsets of $[n]$ given by 
\begin{equation}\label{equ: KI}
K(I)=\{\sigma \subset [n]:L_{\sigma}\not\subset
\bigcup_{\tau\in I}L_{\tau}\}
=
\{\sigma\subset [n]: \tau \not\subset \sigma
\mbox{ for any } \tau\in I\}.
\end{equation}
Then $K(I)$ is a simplicial complex on the index set $[n]$ such that
$U(K(I))=Y_I$ and that
$\dis \bigcup_{\sigma\notin K(I),\sigma\subset [n]}L_{\sigma}=L(I).$%
\footnote{%
The assertion of Lemma \ref{lmm: KI} holds  even if
$\mbox{card}(\sigma)\geq 1$ for any $\sigma \in I$.
}
\qed
\end{lemma}
\begin{remark}\label{rmk: K(I)}
{\rm
(i) Since ${\bf 0}\notin Y_I$, one can easily see that
$K(I)\subset \{\sigma\subset [n]:\sigma\not= [n]\}.$
\par
(ii)
One can represent the space $X_I$ in the form
\begin{equation}\label{equ: XI=U(K)/C}
X_I=U(K(I))/\C^*,
\end{equation}
and  show that  (\ref{equ: XI=U(K)/C})
is a homogeneous coordinate representation of the toric variety $X_I$
(\cite[Chapter 5]{CLS},  Remark \ref{rmk: homogeneous}).
\par
(iii)
We can determine explicitly the fan ${\bf \Sigma}_I$ of the toric variety
$X_I$ in terms of the simplicial complex $K(I)$
(see Proposition \ref{prp: Sigma (I)}).
\qed
}
\end{remark}

\begin{example}\label{exa: KI}
{\rm 
$\I$
If $I=J(n)$, $K(J(n))$ is a simplicial complex consisting of $n$ vertices,
$K(J(n))=\{\emptyset , \{j\}:0\leq j\leq n-1\}.$
\par
$\II$
If $I=I(n)$,
$K(I(n))$ is a simplicial complex consisting of all proper subsets of
$[n]$, $K(I(n))=\partial \Delta^{n-1}:=\{\sigma\subset [n]:\sigma\not= [n]\}.$
}
\qed
\end{example}

\begin{definition}\label{dfn: veeI}
{\rm
Let $I$ be any collection of subsets of $[n]$ and
$(X,*)$  a based space.
Let
$\vee^IX\subset X^n$ denote the subspace
consisting of all $(x_0,\cdots ,x_{n-1})\in X^n$ such that,
for each $\sigma\in I$, $x_j=*$ for some $j\in \sigma$
as in \cite[page 436]{GKY1}.
\par
The space $\vee^IX$ is called \textit{the generalized wedge product} of $X$ of type $I$,
and
it is known that there is a homotopy equivalence
$\Omega^2_dX_I\simeq \Omega^2(\vee^I\CP^{\infty})$
(see Corollary \ref{crl: equivalence}).
}
\end{definition}

\paragraph{\bf 1.3. Spaces of maps.}
\ For connected spaces $X$ and $Y$, let
$\Map(X,Y)$ (resp. $\Map^*(X,Y)$) denote the space
consisting of all continuous maps
(resp. base-point preserving continuous maps) from $X$ to $Y$
with the compact-open topology.
When $X$ and $Y$ are complex manifolds, we denote by
$\Hol (X,Y)$ (resp. $\Hol^*(X,Y))$ the subspace of
$\Map (X,Y)$ (resp. $\Map^*(X,Y))$ consisting of all holomorphic maps
(resp. base-point preserving holomorphic maps).
\par
For each integer $d\geq 1$, let
$\Map_d^*(S^2,X_I)=\Omega^2_dX_I$ denote the space of all
based continuous maps
$f:(S^2,\infty)\to (X_I,[1:\cdots :1])$ such that
$[f]=d\in \Z=\pi_2(X_I)$,%
\footnote{
Note that $X_I$ is simply-connected and $\pi_2(X_I)=\Z$
(for the details see Lemma \ref{lmm: pi2XI}).}
where we identify $S^2=\C \cup \{\infty\}$
and choose $\infty\in S^2$ and
$[1:\cdots :1]\in X_I$ as the base points of $S^2$
and $X_I$, respectively.
Let $\Hol_d^*(S^2,X_I)$ denote the subspace of $\Map_d^*(S^2,X_I)$
consisting of all based holomorphic maps of degree $d$.

\begin{definition}\label{dfn: Hol}
{\rm
(i)
Let $\P^d(\C)$ denote the space consisting of all monic polynomials 
$f(z)=z^d+a_1z^{d-1}+\cdots +a_d\in \C[z]$ of the degree $d$.
Then the space $\Hol_d^*(S^2,X_I)$ can be identified with the space consisting
of all $n$-tuples $(f_0(z),\cdots ,f_{n-1}(z))\in \P^d(\C)^n$
of monic polynomials of the same degree $d$ such that polynomials
$f_{i_1}(z),\cdots ,f_{i_s}(z)$ have no common root for any 
$\sigma =\{i_1,\cdots ,i_s\}\in I$, i.e.
the space $\Hol_d^*(S^2,X_I)$ is also identified with
\begin{equation*}
\big\{(f_0,\cdots ,f_{n-1})\in \P^d(\C)^n:
\{f_j(z)\}_{j\in \sigma}
\mbox{ have no common root for any }\sigma\in I\big\}.
\end{equation*}
(ii)
A map $f:X\to Y$ is called {\it a homotopy equivalence} 
(resp. {\it a homology equivalence}) {\it up to dimension} $D$
if the induced homomorphism
$f_*:\pi_k(X)\to \pi_k(Y)$
(resp. $f_*:H_k(X,\Z)\to H_k(Y,\Z))$
is an isomorphism for any $k<D$ and an epimorphism if $k=D$.
Similarly, it is called
{\it a homotopy equivalence} 
(resp. {\it a homology equivalence}) {\it through dimension} $D$
if 
$
f_*:\pi_k(X)\to \pi_k(Y)$
(resp.
$f_*:H_k(X,\Z)\to H_k(Y,\Z))$
is an isomorphism for any $k\leq D.$ 
\par
(iii)
Let $\rmin (I)$ denote the positive integer defined by}
\begin{equation}\label{equ: rmin}
\rmin (I)=\min \{\mbox{\rm card}(\sigma):\sigma \in I\}.
\end{equation}
\end{definition}

\begin{remark}
{\rm
Note that $\rmin (I)$ is an integer such that $2\leq \rmin (I) \leq n$.
For example,  
 $\rmin (I(n))=n$ and $\rmin (J(n))=2$.
 \qed
 }
\end{remark}
Now recall the following two results given in \cite{GKY1}
and \cite{Se}.

\begin{theorem}[G. Segal, \cite{Se}; The case $I=I(n)$]\label{thm: SE}
The inclusion map
$$
i_d:\Hol_d^*(S^2,\CP^{n-1})\to \Map_d^*(S^2,\CP^{n-1})=\Omega^2_d\CP^{n-1}
\simeq \Omega^2S^{2n-1}
$$
is a homotopy equivalence up to dimension $(2n-3)d$.
\qed
\end{theorem}

\begin{theorem}[\cite{GKY1}]\label{thm: GKY1}
The inclusion map
$$
i_d:\Hol_d^*(S^2,X_I)\to \Map_d^*(S^2,X_I)=\Omega^2_dX_I
\simeq \Omega^2(\vee^I\CP^{\infty})
$$
is a homotopy equivalence up to dimension $d$.
\qed
\end{theorem}

\paragraph{\bf 1.4. The main results.}
The main purpose of this paper is to improve the homotopy stability dimension
of the above two results  
as follows.
\begin{theorem}\label{thm: I}
If $\rmin (I)\geq 3$,
the inclusion map
$$
i_d:\Hol_d^*(S^2,X_I)\to \Map_d^*(S^2,X_I)=\Omega^2_dX_I
\simeq \Omega^2(\vee^I\CP^{\infty})
$$
is a homotopy equivalence through dimension $D(I;d)=(2\rmin (I) -3)d-2$.
\end{theorem}
\begin{theorem}[The case $I=I(n)$]\label{thm: Segal}
If $n\geq 3$,
the inclusion map
$$
i_d:\Hol_d^*(S^2,\CP^{n-1})\to \Map_d^*(S^2,\CP^{n-1})=\Omega^2_d\CP^{n-1}\simeq \Omega^2S^{2n-1}
$$
is a homotopy equivalence through dimension 
$D^*(d,n)=(2n -3)(d+1)-1$.
\end{theorem}
\begin{remark}\label{rmk: 2}
{\rm
(i)
We can prove that the inclusion map $i_d$ (of Theorem \ref{thm: I}) is a homology
equivalence through dimension $D(I;d)$ for $\rmin (I)=2$.
But if $\rmin (I)=2$, 
$D(I;d)=d-2<d$ and  the stability dimension in Theorem \ref{thm: I} is weaker than that in 
Theorem \ref{thm: GKY1}.
\par
(ii)
One can give another proof of Theorem \ref{thm: Segal} by using
 \cite{BM} and \cite{CCMM} 
(see \S \ref{section: Proofs}).
\qed
}
\end{remark}


\par\vspace{2mm}\par
Theorem \ref{thm: I} should be compared with the following result
given in
\cite{MV}.

\begin{theorem}[J. Mostovoy and E. Munguia-Villanueva, \cite{MV}]\label{thm: MV}
Let $\XS$ be a smooth compact toric variety associated to a fan ${\bf \Sigma}$.
 Then, if $1\leq m<r_{\bf \Sigma}$, the inclusion of the space of morphisms  
 $\CP^m \to \XS$ given by homogeneous polynomials of fixed degrees ${d_1,\dots,d_r}$ in Cox's homogeneous coordinates  $(${\rm \cite{C}}$)$ into the space 
$\Omega^m \XS$
is a homology equivalence through dimension
 $D_{*}(d_1,\cdots d_r;m)$, 
where we set
\begin{equation}\label{equ: D-MV}
d=\min\{d_1,\cdots ,d_r\}
\quad\mbox{ and }\quad
D_*(d_1,\cdots d_r;m)=(2r_{\bf \Sigma}-2m-1)d-2.
\end{equation}
Here, $r_{\bf \Sigma}$  denotes the positive integer defined in
$($\ref{equ: rSigma}$)$,
$r$ is the number of one dimensional cones in 
${\bf \Sigma}$,
and $d_1,\dots, d_r$ are the degrees of the homogeneous polynomials representing the morphisms in Cox's homogeneous coordinates.
\qed
\end{theorem}
\begin{remark}\label{rmk: MV result}
{\rm
The space $X_I$ is a non-compact toric variety if  $I\not= I(n)$. 
This means that the Theorem \ref{thm: MV} cannot be applied for the case
$m=1$
when $I\not= I(n)$.  
It is not immediately obvious how the two bounds are related, since the bound in Theorem \ref{thm: MV} is defined in terms of the properties of the fan of a toric variety while the bound in Theorem
\ref{thm: I} is given in terms of the orbits of the torus action.
However, the bounds actually turn out to be
the same
(see Lemma \ref{lmm: rmin=rSigma} and Remark \ref{rmk: dimension}).
Thus our theorem shows that the above result given in Theorem \ref{thm: MV}
holds also for a class of non-compact toric varieties $X_I$, at least for $m=1$.
\par
Note  that there are no non-constant holomorphic map from
$\CP^1$ to $\XS$  for $\XS =\T^n_{\C}$ or $\C^n$.
Because $r_{\bf \Sigma}=0$ or $1$ iff $\XS =\T^n_{\C}$ or $\C$, 
we need the assumption that $r_{\bf \Sigma}\geq 2$.
Moreover,
 because the number $D_*(d_1,\cdots ,d_r;m)$ is 
a negative integer if $r_{\bf \Sigma}\leq m$,
we also need to assume that $r_{\bf \Sigma}>m\geq 1$.
It is natural to ask: is the result of Theorem \ref{thm: MV}
still true 
  for  a non-compact smooth 
  toric variety $\XS$ such that $r_{\bf \Sigma}>m\geq 1$?
\par
It is clear that the proof of Theorem \ref{thm: MV} cannot be used because it relies on the 
Stone-Weierestrass theorem and the fact that for any connected space $E$ and 
a compact Riemannian manifold $X$ there exists an $\epsilon>0$ such that any two maps
  from $E$ to $X$ 
  that are uniformly $\epsilon$-close are homotopic. 
  Since this is clearly not true if $\XS$ is not compact, a different argument is needed. 
  In the case $m=1$ and $\XS =X_I$, we can use a variant of Segal's \lq\lq scanning map\rq\rq\ argument (see \cite{Se}), which we will describe in the next section. \qed  } 
\end{remark}

This paper is organized as follows.
In \S \ref{section: 2} we recall the  \lq\lq stable result\rq\rq  of \cite{GKY1}
(Theorem \ref{thm: stable}) 
and state the unstable results (Theorem \ref{thm: unstable} and Theorem
\ref{thm: CP}).
In \S \ref{section: simplicial resolution} we recall the definitions of simplicial resolutions and
in \S \ref{section: spectral sequence}
 construct the Vassiliev  spectral sequences associated with a non-degenerate simplicial resolution and with the corresponding truncated ones. 
In \S \ref{section: Proofs}, we give proofs of the main unstable results
(Theorem \ref{thm: unstable} and Theorem \ref{thm: CP}) and
prove Theorem \ref{thm: Segal}.
In \S \ref{section: polyhedral product},
we review several basic facts concerning the topology of the space $X_I$
from the point of view of polyhedral products, and prove the existence of a
certain fibration sequence 
(Proposition \ref{lmm: fibration}).
Finally in  \S \ref{section: fan}
we study the relation between the simplicial complex $K(I)$ and
the fan ${\bf \Sigma}_I$, and
give the proof of the main result (Theorem \ref{thm: I}).

\section{Stabilization.}\label{section: 2}

\paragraph{\bf 2.1. Stabilization maps.}
First we review several definitions and basic results  concerning stabilization  obtained in \cite{GKY1}.
\begin{definition}\label{dfn: SP}
{\rm
(i)
Let $S_d$ denote
the symmetric group on $d$ letters.
For a space $X$, $S_d$ acts on $X^d=X\times \cdots \times X$
($d$-times) by permuting coordinates and  
let $\SP^d(X)$ denote the $d$-th symmetric product of $X$ given by
the orbit space 
$\SP^d(X)=X^d/S_d.$
\par
(ii)
Let $F(X,d)\subset X^d$ denote the configuration space
\begin{equation}
F(X,d)=\{(x_0,\cdots ,x_{d-1})\in X^d:
x_i\not= x_j\mbox{ if }i\not= j\}
\end{equation}
of {\it ordered $d$-distinct points in} $X$.
Since it is $S_d$-invariant, we can define a subspace
$C_d(X)\subset \SP^d(X)$ as the orbit space $C_d(X)=F(X,d)/S_d$.
It is usually called {\it the configuration space of unordered
$d$-distinct points in} $X$.}
\end{definition}
\begin{remark}\label{rmk: SP}
{\rm 
Note that
each element $\alpha \in \SP^d(X)$ can be represented as a formal sum
$\alpha =\sum_{k=0}^rn_kx_k$,
where $\{x_k\}_{k=0}^r$ are mutually distinct points in $X$ and each $n_k$ is a positive
integer such that $\sum_{k=0}^rn_k=d$.}
\qed
\end{remark}
\begin{definition}\label{dfn: EI}
{\rm
For each space $X$, let
 $E_d^I(X)$ denote the space defined by}
\begin{equation}\label{equ: EI}
E_d^I(X)=\{ (\xi_0,\cdots ,\xi_{n-1})\in \SP^d(X)^n:\cap_{j\in \sigma}\xi_j=\emptyset
\mbox{ for any }\sigma\in I\}.
\end{equation}
{\rm
Note that
one can  identify $\P^d(\C)$ with the space
$\SP^d(\C)$ by the correspondence
$\prod_{k=0}^r(z-x_k)^{n_k}\mapsto
\sum_{k=0}^rn_kx_k.$
It is also easy to see that, with this identification,
there is a natural homeomorphism
\begin{equation}\label{equ: H=E}
\Hol_d^*(S^2,X_I)\cong E_d^I(\C).
\end{equation}
Since there is a homeomorphism $\C\cong U_d:=\{x\in \C:\mbox{Re}(x) <d\}$,
by choosing  any $n$ distinct points $\{x_{k;d}\}_{k=0}^{n-1}$ with 
$d<\mbox{Re}(x_{k;d}) <d+1$, one can define
the  map
\begin{equation}\label{equ; sd}
\hat{s}_d:E_d^I(U_d)\to E^I_{d+1}(U_{d+1})
\end{equation}
by
$
\hat{s}_d(\xi_0,\cdots ,\xi_{n-1})=(\xi_0+x_{0;d},\cdots ,\xi_{n-1}+x_{n-1;d}).
$
Then {\it the stabilization map}
\begin{equation}\label{equ; sta}
s_d:\Hol^*_d(S^2,X_I)\to \Hol^*_{d+1}(S^2,X_I)
\end{equation}
is defined as the composite of maps}
$$
\Hol^*_d(S^2,X_I)\cong E^I_d(\C) \cong E_d^I(U_d) 
\stackrel{\hat{s}_d}{\longrightarrow} 
E_{d+1}^I(U_{d+1})\cong E^I_{d+1}(\C)
\cong \Hol^*_{d+1}(S^2,X_I).
$$
\end{definition}
\begin{remark}\label{rmk: sta1}
{\rm
(i) 
Note that, while the map $s_d$ depends on the choice of the points
$\{x_{k;d}\}_{k=0}^{n-1}$,  its homotopy class 
does not.}
\par
{\rm
(ii)
If we choose a sufficiently small number $\epsilon >0$
and denote by
$V_k$ $(0\leq k\leq n-1)$ the open disk of radius $\epsilon$
with center $x_{k;d}$, then we may suppose that 
$V_k\cap V_j=\emptyset$ if $k\not=j$ and that
$V_k\subset U_{d+1}\setminus \overline{U}_d$ for each $0\leq k\leq n-1$.
\par
In this situation, the map $\hat{s}_d:E_d^I(U_d)\to E^I_{d+1}(U_{d+1})$
extends to the open embedding
\begin{equation}\label{equ: emb sd}
\hat{s}_d^{\p}:V_0\times V_1\times\cdots \times V_{n-1}\times E_d^I(U_d)\to E^I_{d+1}(U_{d+1})
\end{equation}
by
$\hat{s}_d^{\p}(x_0,\cdots ,x_{n-1},\xi_0,\cdots ,\xi_{n-1})=
(\xi_0+x_0,\cdots ,\xi_{n-1}+x_{n-1})$.
Since there is a homeomorphism $\C\cong V_k$,
the stabilization map
$s_d$ naturally extends to the open embedding}
\begin{equation}\label{equ: open-sd}
\overline{s}_d:\C^n\times \Hol_d^*(S^2,X_I)\to \Hol_{d+1}^*(S^2,X_I),
\end{equation} 
which will be used in  \S \ref{section: spectral sequence}
(see (\ref{equ: sssd})).
\end{remark}

\paragraph{\bf 2.2. Homological stability. }
\ Recall the following result which can be easily proved by using
Theorem \ref{thm: GKY1}.
\begin{theorem}[\cite{GKY1}]\label{thm: stable}
$\dis
i_{\infty}=\lim_{d\to\infty}i_d:\lim_{d\to\infty}\Hol_d^*(S^2,X_I)
\stackrel{\simeq}{\longrightarrow}
\lim_{d\to\infty}\Omega^2_dX_I\simeq \Omega^2_0X_I
$
is a homotopy equivalence.
\qed
\end{theorem}
\begin{remark}\label{rmk: GKY1}
{\rm
The proof, of 
this result (Theorem \ref{thm: stable}), whose details were omitted in \cite{GKY1},
requires only  the use of Graeme Segal's  well known  \lq\lq scanning map\rq\rq\   method
(eg. \cite{Gu}, \cite{Gu2}, \cite{GKY1}, \cite{GKY2}, \cite{Se})
and the existence of a homotopy equivalence
\begin{equation}\label{equ: Segal equivalence}
\Omega^2_0X_I\simeq \Omega^2(\vee^I\CP^{\infty}).
\end{equation}
In \cite{GKY1}
we originally obtained the above homotopy equivalence
by using Segal's fibration sequence  \cite[\S 2]{Se}.
However, for the sake of completeness 
we will give another proof of (\ref{equ: Segal equivalence})
in \S \ref{section: polyhedral product} (Corollary \ref{crl: equivalence}).
\qed
}
\end{remark}
The key steps in the proofs of the main results of this paper are the following theorems:
\begin{theorem}\label{thm: unstable}
The stabilization map
$s_d:\Hol_d^*(S^2,X_I)\to \Hol_{d+1}^*(S^2,X_I)$ is
a homology equivalence through dimension
$D(I;d).$
\end{theorem}

\begin{theorem}\label{thm: CP}
If $I= I(n)$,
the stabilization map
$$
s_d:\Hol_d^*(S^2,\CP^{n-1})\to \Hol_{d+1}^*(S^2,\CP^{n-1})
$$ is
a homology equivalence through dimension
$D^*(d,n)=(2n -3)(d+1)-1.$
\end{theorem}



Proofs of Theorem \ref{thm: unstable}
and Theorem \ref{thm: CP}
 are given in \S \ref{section: Proofs}.

\section{Simplicial resolutions.}\label{section: simplicial resolution}

In this section, we give the definitions of and summarize the basic facts about  non-degenerate simplicial resolutions  
and the associated truncated simplicial resolutions of maps (\cite{AKY1}, \cite{KY4}, \cite{Mo2}, \cite{Mo3}, \cite{Va}).
\begin{definition}\label{def: def}
{\rm
(i) For a finite set $\textbf{\textit{v}} =\{v_1,\cdots ,v_l\}\subset \R^N$,
let $\sigma (\textbf{\textit{v}})$ denote the convex hull spanned by 
$\textbf{\textit{v}}.$
Let $h:X\to Y$ be a surjective map such that
$h^{-1}(y)$ is a finite set for any $y\in Y$, and let
$i:X\to \R^N$ be an embedding.
Let  $\mathcal{X}^{\Delta}$  and $h^{\Delta}:{\mathcal{X}}^{\Delta}\to Y$ 
denote the space and the map
defined by
\begin{equation}
\mathcal{X}^{\Delta}=
\big\{(y,u)\in Y\times \R^N:
u\in \sigma (i(h^{-1}(y)))
\big\}\subset Y\times \R^N,
\ h^{\Delta}(y,u)=y.
\end{equation}
The pair $(\mathcal{X}^{\Delta},h^{\Delta})$ is called
{\it the simplicial resolution of }$(h,i)$.
In particular, $(\mathcal{X}^{\Delta},h^{\Delta})$
is called {\it a non-degenerate simplicial resolution} if for each $y\in Y$
any $k$ points of $i(h^{-1}(y))$ span $(k-1)$-dimensional simplex of $\R^N$.
\par
(ii)
For each $k\geq 0$, let $\mathcal{X}^{\Delta}_k\subset \mathcal{X}^{\Delta}$ be the subspace
given by 
\begin{equation}
\mathcal{X}_k^{\Delta}=\big\{(y,u)\in \mathcal{X}^{\Delta}:
u \in\sigma (\textbf{\textit{v}}),
\textbf{\textit{v}}=\{v_1,\cdots ,v_l\}\subset i(h^{-1}(y)),\ l\leq k\big\}.
\end{equation}
We make identification $X=\mathcal{X}^{\Delta}_1$ by identifying 
 $x\in X$ with the pair
$(h(x),i(x))\in \mathcal{X}^{\Delta}_1$,
and we note that  there is an increasing filtration
\begin{equation}\label{equ: filtration}
\emptyset =
\mathcal{X}^{\Delta}_0\subset X=\mathcal{X}^{\Delta}_1\subset \mathcal{X}^{\Delta}_2\subset
\cdots \subset \mathcal{X}^{\Delta}_k\subset \mathcal{X}^{\Delta}_{k+1}\subset
\cdots \subset \bigcup_{k= 0}^{\infty}\mathcal{X}^{\Delta}_k=\mathcal{X}^{\Delta}.
\end{equation}
}
\end{definition}

\begin{lemma}[\cite{Mo2}, \cite{Mo3}, \cite{Va}]\label{lemma: simp}
Let $h:X\to Y$ be a surjective map such that
$h^{-1}(y)$ is a finite set for any $y\in Y,$ 
$i:X\to \R^N$ an embedding, and let
$(\mathcal{X}^{\Delta},h^{\Delta})$ denote the simplicial resolution of $(h,i)$.
\par
\begin{enumerate}
\item[$\I$]
If $X$ and $Y$ are semi-algebraic spaces and the
two maps $h$, $i$ are semi-algebraic maps, then
$h^{\Delta}:\mathcal{X}^{\Delta}\stackrel{\simeq}{\rightarrow}Y$
is a homotopy equivalence.
Moreover,
there is an embedding $j:X\to \R^M$ such that
the associated simplicial resolution
$(\tilde{\mathcal{X}}^{\Delta},\tilde{h}^{\Delta})$ of $(h,j)$
is non-degenerate.
\par
\item[$\II$]
If there is an embedding $j:X\to \R^M$ such that its associated simplicial resolution
$(\tilde{\mathcal{X}}^{\Delta},\tilde{h}^{\Delta})$
is non-degenerate,
the space $\tilde{\mathcal{X}}^{\Delta}$
is uniquely determined up to homeomorphism and
there is a filtration preserving homotopy equivalence
$q^{\Delta}:\tilde{\mathcal{X}}^{\Delta}\stackrel{\simeq}{\rightarrow}{\mathcal{X}}^{\Delta}$ such that $q^{\Delta}\vert X=\mbox{id}_X$.
\qed
\end{enumerate}
\end{lemma}

\begin{remark}\label{Remark: non-degenerate}
{\rm
Even when a  surjective map $h:X\to Y$ is not finite to one,  
it is still possible to construct an associated non-degenerate simplicial resolution.
It is easy to show  (see \cite{Va} chapter III) that there exists a sequence of embeddings
$\{\tilde{i}_k:X\to \R^{N_k}\}_{k\geq 1}$ satisfying the following two conditions
for each $k\geq 1$:
\begin{enumerate}
\item[({\ref{equ: filtration}}$)_k$]
\begin{enumerate}
\item[(i)]
For any $y\in Y$,
any $t$ points of the set $\tilde{i}_k(h^{-1}(y))$ span a $(t-1)$-dimensional affine subspace
of $\R^{N_k}$ if $t\leq 2k$.
\item[(ii)]
$N_k\leq N_{k+1}$ and if we identify $\R^{N_k}$ with a subspace of
$\R^{N_{k+1}}$, 
then $\tilde{i}_{k+1}=\hat{i}\circ \tilde{i}_k$,
where
$\hat{i}:\R^{N_k}\stackrel{\subset}{\rightarrow} \R^{N_{k+1}}$
denotes the inclusion.
\end{enumerate}
\end{enumerate}
A general non-degenerate simplicial resolution may be
constructed by choosing a sequence of embeddings
$\{\tilde{i}_k:X\to \R^{N_k}\}_{k\geq 1}$ satisfying the above two conditions
for each $k\geq 1$.
We then let
$\dis\mathcal{X}^{\Delta}_k=\big\{(y,u)\in Y\times \R^{N_k}:
u\in\sigma (\textbf{\textit{v}}),
\textbf{\textit{v}}
=\{v_1,\cdots ,v_l\}\subset \tilde{i}_k(h^{-1}(y)),l\leq k\big\}.$
Identifying naturally  $\mathcal{X}^{\Delta}_k$ with a subspace
of $\mathcal{X}_{k+1}^{\Delta}$,  we can define the non-degenerate simplicial
resolution $\mathcal{X}^{\Delta}$ of  $h$ as the union  
$\dis \mathcal{X}^{\Delta}=\bigcup_{k\geq 1} \mathcal{X}^{\Delta}_k$.
}
\qed
\end{remark}



\begin{definition}\label{def: 2.3}
{\rm
Let $h:X\to Y$ be a surjective semi-algebraic map between semi-algebraic spaces, 
$j:X\to \R^N$ be a semi-algebraic embedding, and let
$(\mathcal{X}^{\Delta},h^{\Delta}:\mathcal{X}^{\Delta}\to Y)$
denote the associated non-degenerate  simplicial resolution of $(h,j)$. 
\par
Let $k$ be a fixed positive integer and let
$h_k:\mathcal{X}^{\Delta}_k\to Y$ be the map
defined by the restriction
$h_k:=h^{\Delta}\vert \mathcal{X}^{\Delta}_k$.
The fibres of the map $h_k$ are $(k-1)$-skeleta of the fibres of $h^{\Delta}$ and, in general,  always
fail to be simplices over the subspace
$Y_k=\{y\in Y:\mbox{card}(h^{-1}(y))>k\}.$
Let $Y(k)$ denote the closure of the subspace $Y_k$.
We modify the subspace $\mathcal{X}^{\Delta}_k$ so as to make  all
the fibres of $h_k$ contractible by adding to each fibre of $Y(k)$ a cone whose base
is this fibre.
We denote by $X^{\Delta}(k)$ this resulting space and by
$h^{\Delta}_k:X^{\Delta}(k)\to Y$ the natural extension of $h_k$.
\par
Following  \cite{Mo3}, we call the map $h^{\Delta}_k:X^{\Delta}(k)\to Y$
{\it the truncated $($after the $k$-th term$)$  simplicial resolution} of $Y$.
Note that 
that there is a natural filtration
$$
\emptyset =X^{\Delta}_0\subset X^{\Delta}_1\subset
\cdots 
\subset X^{\Delta}_l\subset X^{\Delta}_{l+1}\subset \cdots
\subset  X^{\Delta}_k\subset X^{\Delta}_{k+1}
=X^{\Delta}_{k+2}
=\cdots =X^{\Delta}(k),
$$
where $X^{\Delta}_l=\mathcal{X}^{\Delta}_l$ if $l\leq k$ and
$X^{\Delta}_l=X^{\Delta}(k)$ if $l>k$.
}
\end{definition}

\begin{lemma}[\cite{Mo3} (cf. \cite{KY4}, Remark 2.4 and Lemma 2.5)]\label{Lemma: truncated}
Under the same assumptions and notations as in Definition \ref{def: 2.3}, the map
$h^{\Delta}_k:X^{\Delta}(k)\stackrel{\simeq}{\longrightarrow} Y$ is a homotopy equivalence.
\qed
\end{lemma}

\section{The Vassiliev spectral sequence.}\label{section: spectral sequence}

\paragraph{\bf 4.1. Vassiliev spectral sequences.}
\ 
In this section,
we identify $\Hol_d^*(S^2,X_I)$ with the space 
of all
$n$-tuples
$(f_0(z),\cdots ,f_{n-1}(z))\in \P^d(\C)^n$
of monic polynomials of the same degree $d$
such that
$f_{i_1}(z),\cdots ,f_{i_s}(z)$
have no common root for any
$\sigma =\{i_1,\cdots ,i_s\}\in I$
as in Definition \ref{dfn: Hol}.

\begin{definition}\label{Def: 3.1}
{\rm
\par
(i)
By the  \emph{the discriminant} $\Sigma_d$ of $\Hol_d^*(S^2,X_I)$ in $\P^d(\C)^n$ we mean
the complement
\begin{eqnarray*}
\Sigma_d
&=&
\P^d(\C)^n \setminus \Hol_d^*(S^2,X_I)
\\
&=&
\{(f_0,\cdots ,f_{n-1})\in \P^d(\C)^n :
(f_0(x),\cdots ,f_{n-1}(x))\in L(I)
\mbox{ for some }x\in \C\}.
\end{eqnarray*}
\par
(ii)
Let  $Z_d\subset \Sigma_d\times \C$
denote 
{\it the tautological normalization} of 
 $\Sigma_d$
consisting of all pairs 
$(F,x)=((f_0,\ldots ,f_{n-1}),
x)\in \Sigma_d\times\C$
such that 
$(f_0(x),\cdots ,f_{n-1}(x))\in L(I)$.
Projection on the first factor  gives a surjective map
$\pi_d :Z_d\to \Sigma_d$
}
\end{definition}

Our goal in this section is to construct, by means of the
{\it non-degenerate} simplicial resolution  of the discriminant, a spectral sequence converging to the homology of
$\Hol_d^*(S^2,X_I)$.

\begin{definition}\label{non-degenerate simp.}
{\rm
Let 
$(\SZ,{\pi}^{\Delta}_d:\SZ\to\Sigma_d)$ 
be the non-degenerate simplicial resolution associated to the surjective map
$\pi_d:Z_d\to \Sigma_d$ 
with the natural increasing filtration as in Definition \ref{def: def},
$$
\emptyset =
\SZ_0
\subset \SZ_1\subset 
\SZ_2\subset \cdots
\subset 
\SZ=\bigcup_{k= 0}^{\infty}\SZ_k.
$$
}
\end{definition}


By Lemma \ref{lemma: simp},
the map
$\pi_d^{\Delta}:
\SZ\stackrel{\simeq}{\rightarrow}\Sigma_d$
is a homotopy equivalence which
extends to  a homotopy equivalence
$\pi_{d+}^{\Delta}:\SZ_+\stackrel{\simeq}{\rightarrow}{\Sigma_{d+}},$
where $X_+$ denotes the one-point compactification of a
locally compact space $X$.
Since
${\mathcal{X}_k^{d}}_+/{\SZ_{k-1}}_+
\cong (\SZ_k\setminus \SZ_{k-1})_+$,
we have a spectral sequence 
$$
\big\{E_{t;d}^{k,s},
d_t:E_{t;d}^{k,s}\to E_{t;d}^{k+t,s+1-t}
\big\}
\Rightarrow
H^{k+s}_c(\Sigma_d,\Z),
$$
where
$E_{1;d}^{k,s}=\tilde{H}^{k+s}_c(\SZ_k\setminus\SZ_{k-1},\Z)$ and
$H_c^k(X,\Z)$ denotes the cohomology group with compact supports given by 
$
H_c^k(X,\Z)= H^k(X_+,\Z).
$
\par
Since there is a homeomorphism
$\P^d(\C)^n\cong \C^{dn}$,
by Alexander duality  there is a natural
isomorphism
\begin{equation}\label{Al}
\tilde{H}_k(\Hol_d^*(S^2,X_I),\Z)\cong
\tilde{H}_c^{2nd-k-1}(\Sigma_d,\Z)
\quad
\mbox{for any }k.
\end{equation}
By
reindexing we obtain a
spectral sequence
\begin{eqnarray}\label{SS}
&&\big\{E^{t;d}_{k,s}, \tilde{d}^{t}:E^{t;d}_{k,s}\to E^{t;d}_{k+t,s+t-1}
\big\}
\Rightarrow H_{s-k}(\Hol_d^*(S^2,X_I),\Z),
\end{eqnarray}
where
$E^{1;d}_{k,s}=
\tilde{H}^{2nd+k-s-1}_c(\SZ_k\setminus\SZ_{k-1},\Z).$
\par\vspace{2mm}\par
Let
 $L_{k;I}\subset (\C\times L(I))^k$ denote the subspace
defined by
$$
L_{k;I}=\{((x_1,s_1),\cdots ,(x_k,s_k)): 
x_j\in \C,s_j\in L(I),
x_l\not= x_j\mbox{ if }l\not= j\}.
$$
The symmetric group $S_k$ on $k$ letters  acts on $L_{k;I}$ by permuting coordinates. Let
$C_{k;I}$ denote the orbit space
\begin{equation}\label{Ck}
C_{k;I}=L_{k;I}/S_k.
\end{equation}
Note that $C_{k;I}$ is a cell-complex of  dimension
$2(1+n-r_{\rm min}(I))k$.


\begin{lemma}\label{lemma: vector bundle*}
If  
$1\leq k\leq d$,
$\SZ_k\setminus\SZ_{k-1}$
is homeomorphic to the total space of a real affine
bundle $\xi_{d,k}$ over $C_{k;I}$ with rank 
$l_{d,k}=2n(d-k)+k-1$.
\end{lemma}
\begin{proof}
The argument is exactly analogous to the one in the proof of  
\cite[Lemma 4.4]{AKY1}. 
Namely, an element of $\SZ_k\setminus\SZ_{k-1}$ is represented by 
$(F,u)=((f_0,\cdots ,f_{n-1}),u)$, where 
$F=(f_0,\cdots ,f_{n-1})$ is an 
$n$-tuple of polynomials in $\Sigma_d$ and $u$ is an element of the interior of
the span of the images of $k$ distinct points 
$\{x_1,\cdots, x_k\}\in C_k(\C)$ 
such that
$F(x_j)=(f_0(x_j),\cdots ,f_{n-1}(x_j))\in L(I)$ for each $1\leq j\leq k$, 
under a suitable embedding.
\ 
Since the $k$ distinct points $\{x_j\}_{j=1}^k$ 
are uniquely determined by $u$, by the definition of the non-degenerate simplicial resolution (cf.  ({\ref{equ: filtration}}$)_k$),
 there are projection maps
$\pi_{k,d} :\mathcal{X}^{d}_k\setminus
\mathcal{X}^{d}_{k-1}\to C_{k;I}$
defined by
$((f_0,\cdots ,f_{n-1}),u) \mapsto 
\{(x_1,F(x_1)),\dots, (x_k,F(x_k))\}$. 

\par
Now suppose that $1\leq k\leq d$.
Let $c=\{(x_j,s_j)\}_{j=1}^k\in C_{k;I}$
$(x_j\in \C$, $s_j\in L(I))$ be any fixed element and consider the fibre  $\pi_{k,d}^{-1}(c)$.
For each $1\leq j\leq k$,
we set $s_j=(s_{1,j},\cdots ,s_{n,j})$ and
consider the condition  
\begin{equation}\label{equ: pik}
F(x_j)=(f_0(x_j),\cdots ,f_{n-1}(x_j))=s_j
\quad
\Leftrightarrow
\quad
f_t(x_j)=s_{t,j}
\quad
\mbox{for }0\leq t\leq n-1.
\end{equation}
In general, 
the condition $f_t(x_j)=s_{t,j}$ gives
one  linear condition on the coefficients of $f_t$,
and it determines an affine hyperplane in $\P^d(\C)$. 
For example, if we set $f_t(z)=z^d+\sum_{i=0}^{d-1}a_{i,t}z^{i}$,
then
$f_t(x_j)=s_{t,j}$ for any $1\leq j\leq k$
if and only if
\begin{equation}\label{equ: matrix equation}
\begin{bmatrix}
1 & x_1 & x_1^2 & \cdots & x_1^{d-1}
\\
1 & x_2 & x_2^2 & \cdots & x_2^{d-1}
\\
\vdots & \ddots & \ddots & \ddots & \vdots
\\
1 & x_k & x_k^2 & \cdots & x_k^{d-1}
\end{bmatrix}
\cdot
\begin{bmatrix}
a_{0,t}\\ a_{1,t} \\ \vdots 
\\ a_{d-1,t}
\end{bmatrix}
=
\begin{bmatrix}
s_{t,1}-x_1^d\\ s_{t,2}-x_2^d \\ \vdots 
\\ s_{t,k}-x_k^d
\end{bmatrix}
\end{equation}
Since $1\leq k\leq d$ and
 $\{x_j\}_{j=1}^k\in C_k(\C)$,  
it follows from the properties of Vandermonde matrices that the condition (\ref{equ: matrix equation}) 
gives exactly $k$ independent conditions on the coefficients of $f_t(z)$.
Thus  the space of polynomials $f_t(z)$ in $\P^d(\C)$ which satisfies
(\ref{equ: matrix equation})
is the intersection of $k$ affine hyperplanes in general position
and has codimension $k$ in $\P^d(\C)$.
Hence,
the fibre $\pi_{k,d}^{-1}(c)$ is homeomorphic  to the product of an open $(k-1)$-simplex
 with the real affine space of dimension
 $2n(d-k)$.
Since one can show that the local triviality holds,
$\pi_{k,d}$ is a real affine bundle over $C_{k;I}$ of rank $l_{d,k}
=2n(d-k)+k-1$.
\end{proof}

\begin{remark}\label{rmk: CP}
{\rm
Suppose that $I=I(n)$. Then  $L(I)=\{{\bf 0}\}$ and
$s_{t,1}=\cdots =s_{t,k}=0$ in  (\ref{equ: matrix equation}).
Hence, if  (\ref{equ: matrix equation}) is satisfied,
$\{x_j\}_{j=1}^k$ is a set of $k$ distinct roots of $f_t(z)$ for any $0\leq t\leq n-1$.
Because each $f_t(z)$ is a monic polynomial of degree $d$, we must have $k\leq d$. 
Hence, if $I=I(n)$ and $k\geq d+1$, $\SZ_k\setminus\SZ_{k-1}=\emptyset$ and thus
$E^{1;d}_{k,s}=0$ for any  $k\geq d+1$. }
\qed
\end{remark}

\begin{lemma}\label{lemma: E11}
If $1\leq k\leq  d$, there is a natural isomorphism
$$
E^{1;d}_{k,s}\cong
\tilde{H}^{2nk-s}_c(C_{k;I},\pm \Z),
$$
where 
the twisted coefficients system $\pm \Z$  comes from
the Thom isomorphism.
\end{lemma}
\begin{proof}
Suppose that $1\leq k\leq d$.
By Lemma \ref{lemma: vector bundle*}, there is a
homeomorphism
$
(\SZ_k\setminus\SZ_{k-1})_+\cong T(\xi_{d,k}),
$
where $T(\xi_{d,k})$ denotes the Thom space of
$\xi_{d,k}$.
Since $(2nd+k-s-1)-l_{d,k}
=
2nk-s,$
by using the Thom isomorphism 
there is a natural isomorphism 
$
E^{1;d}_{k,s}
\cong 
\tilde{H}^{2nd+k-s-1}(T(\xi_{d,k}),\Z)
\cong
\tilde{H}^{2nk-s}_c(C_{k;I},\pm \Z).
$
\end{proof}

\paragraph{\bf 4.2. Homomorphisms of spectral sequences.}
\ 
Recall $U_d=\{w\in \C:\mbox{Re}(w)<d\}$ and 
the stabilization map
$s_d:\Hol_d^*(S^2,X_I)\to \Hol^*_{d+1}(S^2,X_I)$
given by (\ref{equ; sta}).
Note that the map $s_d$ extends to an open embedding
\begin{equation}\label{equ: sssd}
\overline{s}_d:\C^n\times \Hol_d^*(S^2,X_I)\to
\Hol_{d+1}^*(S^2,X_I)
\end{equation}
as in
(\ref{equ: open-sd}).
In fact, by exactly the same argument, it also
naturally extends to an open embedding
$\tilde{s}_d:\C^n\times \P^d(\C)^n\to \P^{d+1}(\C)^n$ and  by  restriction  we get an open embedding
\begin{equation}\label{equ: open embedding}
\tilde{s}_d:\C^n\times \Sigma_d\to \Sigma_{d+1}.
\end{equation}
Because one-point compactification is contravariant for open embeddings,
(\ref{equ: open embedding}) induces the map
$\tilde{s}_{d+}:\Sigma_{d+1+}\to
(\C^d\times \Sigma_d)_+=S^{2d}\wedge \Sigma_{d+}.$
Note that there is a commutative diagram
\begin{equation}
\begin{CD}
\tilde{H}_k(\Hol_d^*(S^2,X_I),\Z) @>{s_d}_*>>\tilde{H}_k(\Hol_{d+1}^*(S^2,X_I),\Z)
\\
@V{Al}V{\cong}V @V{Al}V{\cong}V
\\
\tilde{H}^{2dn-k-1}_c(\Sigma_d,\Z)
@>{\tilde{s}_{d+}}^{\ *}>>
\tilde{H}^{2(d+1)n-k-1}_c(\Sigma_{d+1},\Z)
\end{CD}
\end{equation}
where $Al$ denotes the Alexander duality isomorphism and
 ${\tilde{s}_{d+}}^{\ *}$ the composite of the
homomorphisms 
$$
\tilde{H}^{2dn-k-1}_c(\Sigma_d,\Z)
\stackrel{\cong}{\rightarrow}
\tilde{H}^{2(d+1)n-k-1}_c(\C^n\times \Sigma_d,\Z)
\stackrel{(\tilde{s}_{d+})^*}{\longrightarrow}
\tilde{H}^{2(d+1)n-k-1}_c(\Sigma_{d+1},\Z).
$$
By the universality of the non-degenerate simplicial resolution
(\cite[Page 286-287]{Mo2}), 
the map $\tilde{s}_d$ (given by (\ref{equ: open embedding}))
also
naturally extends to a filtration preserving open embedding
\begin{equation}\label{equ: flitr-preserve map}
\tilde{s}_d:\C^n \times \SZ \to \SZd.
\end{equation}
This  induces a filtration preserving map
$(\tilde{s}_d)_+:\SZd_+\to (\C^n \times \SZ)_+=S^{2n}\wedge \SZ_+$,
and thus a homomorphism of spectral sequences
\begin{equation}\label{equ: theta1}
\{ \tilde{\theta}_{k,s}^t:E^{t;d}_{k,s}\to E^{t;d+1}_{k,s}\},
\end{equation}
where 
for $\epsilon \in \{0,1\}$,
$E^{1;d+\epsilon}_{k,s}=
\tilde{H}_c^{2n(d+\epsilon )+k-1-s}(\SZ_k\setminus \SZ_{k-1},\Z)$
and
\begin{equation}\label{equ: spectral sequ1}
\big\{E^{t;d+\epsilon}_{k,s}, \tilde{d}^{t}:
E^{t;d+\epsilon}_{k,s}\to E^{t;d+\epsilon}_{k+t,s+t-1}
\big\}
\quad\Rightarrow 
\quad H_{s-k}(\Hol_{d+\epsilon}^*(S^2,X_I),\Z).
\end{equation}
\begin{lemma}\label{lmm: E1}
If $1\leq k\leq d$, $\tilde{\theta}^1_{k,s}:E^{1;d}_{k,s}\to E^{1;d+1}_{k,s}$ is
an isomorphism for any $s$.
\end{lemma}
\begin{proof}
Suppose that $1\leq k\leq d$.
It follows from the proof of Lemma \ref{lemma: vector bundle*}
that there is a homotopy commutative diagram of affine vector bundles
$$
\begin{CD}
\C^n\times (\SZ_k\setminus\SZ_{k-1}) @>>> C_{k;I}
\\
@VVV \Vert @.
\\
\SZd_k\setminus \SZd_{k-1} @>>> C_{k;I}
\end{CD}
$$
Hence, 
we have 
a commutative diagram
$$
\begin{CD}
E^{1,d}_{k,s} @>>\cong> \tilde{H}^{2nk-s}_c(C_{k;I},\pm \Z)
\\
@V{\tilde{\theta}_{k,s}^1}VV \Vert @.
\\
E^{1,d+1}_{k,s} @>>\cong> \tilde{H}^{2nk-s}_c(C_{k;I},\pm \Z)
\end{CD}
$$
and the assertion follows.
\end{proof}
\paragraph{\bf 4.3. Truncated spectral sequences. }
Now we consider the spectral sequences induced by the 
truncated simplicial resolutions.
\begin{definition}
{\rm
Let $X^{\Delta}(d)$ denote the truncated (after the $d$-th term) simplicial resolution of $\Sigma_d$
with the natural filtration
$
\emptyset =X^{\Delta}_0\subset
X^{\Delta}_1\subset \cdots\subset
X^{\Delta}_d\subset X^{\Delta}_{d+1}=X^{\Delta}_{d+2}=\cdots =X^{\Delta}(d),
$
where $X^{\Delta}_k=\SZ_k$ if $k\leq d$ and $X^{\Delta}_k=X^{\Delta}(d)$ if $k\geq d+1$.
\par
Similarly,
let $Y^{\Delta}(d)$ denote  truncated (after the $d$-th term) simplicial resolution of 
$\Sigma_{d+1}$
with the natural filtration
$
\emptyset =Y^{\Delta}_0\subset
Y^{\Delta}_1\subset \cdots\subset
Y^{\Delta}_d\subset Y^{\Delta}_{d+1}=Y^{\Delta}_{d+2}=\cdots =Y^{\Delta}(d),
$
where $Y^{\Delta}_k=\SZd_k$ if $k\leq d$ and $Y^{\Delta}_k=Y^{\Delta}(d)$ if $k\geq d+1$.
\par\vspace{1mm}\par
Then, by using Lemma \ref{Lemma: truncated} and the same method
as in \cite[\S 2 and \S 3]{Mo3} (cf. \cite[Lemma 2.2]{KY4}), 
for $\epsilon\in \{0,1\}$ we obtain {\it the truncated spectral sequence}
\begin{equation*}\label{equ: spectral sequ2}
\big\{E(d+\epsilon)^{t}_{k,s}, d^{t}:E(d+\epsilon)^{t}_{k,s}\to 
E(d+\epsilon)^{t}_{k+t,s+t-1}
\big\}
\Rightarrow H_{s-k}(\Hol_{d+\epsilon}^*(S^2,X_I),\Z),
\end{equation*}
where
$$
\begin{cases}
\E^{1}_{k,s}&=\ \tilde{H}_c^{2nd+k-1-s}(X^{\Delta}_k\setminus X^{\Delta}_{k-1},\Z),
\\
\Ed^{1}_{k,s}&=\ \tilde{H}_c^{2n(d+1)+k-1-s}(Y^{\Delta}_k\setminus Y^{\Delta}_{k-1},\Z).
\end{cases}
$$ 
}
\end{definition}
\begin{remark}
{\rm
Note that the notation $\Ed^{t}_{k,s}$ refers to the spectral sequence of the simplicial resolution of $\SZd$ truncated after the $d$-th  rather than $d+1$-th term.}
\qed
\end{remark}
By the naturality of truncated simplicial resolutions,
the filtration preserving map
$\tilde{s}_d:\C^n\times \SZ \to \SZd$
(given by (\ref{equ: flitr-preserve map}))  gives rise to a natural filtration preserving map
\begin{equation}
\tilde{s}_d^{\p}:\C^n\times X^{\Delta}(d) \to Y^{\Delta}(d)
\end{equation}
which, in a way analogous to  (\ref{equ: theta1}), induces
a homomorphism of spectral sequences 
\begin{equation}\label{equ: theta2}
\{ \theta_{k,s}^t:\E^{t}_{k,s}\to \Ed^{t}_{k,s}\}.
\end{equation}
\begin{lemma}\label{lmm: Ed}
Let $r_{\rm min}=r_{\rm min}(I)$, and let $\epsilon \in\{0,1\}$.
\begin{enumerate}
\item[$\I$]
If $k<0$ or $k\geq d+2$,
$E(d+\epsilon )^1_{k,s}=0$ for any $s$.
\item[$\II$]
$E(d+\epsilon )^1_{0,0}=\Z$ and $E(d+\epsilon)^1_{0,s}=0$ if $s\not= 0$.
\item[$\III$]
If $1\leq k\leq d$, there is a natural isomorphism
$E(d+\epsilon)^1_{k,s}\cong\tilde{H}^{2nk-s}_c(C_{k;I},\pm \Z)$.
\item[$\IV$]
$E(d+\epsilon)^1_{d+1,s}=0$ for any $s\leq (2\rmin -2)d-1$.
\end{enumerate}
\end{lemma}
\begin{proof}
Since the proofs of both cases are identical,  we will only consider the case
$\epsilon =0$.
Since $X^{\Delta}_k=\SZ_k$ for any $k\geq d+2$,
the assertions (i) and (ii) are clearly true.
Since $X^{\Delta}_k=\SZ_k$ for any $k\leq d$,
(iii) easily follows from Lemma \ref{lemma: E11}.
Thus it remains to prove (iv).
By Lemma \cite[Lemma 2.1]{Mo3},
$$
\dim (X^{\Delta}_{d+1}\setminus X^{\Delta}_d)=\dim (\SZ_d\setminus \SZ_{d-1})+1
=(l_{d,d}+\dim C_{d;I})+1=2nd+3d-2\rmin d.
$$
Because 
$\E^1_{d+1,s}=\tilde{H}_c^{2nd+d-s}(X^{\Delta}_{d+1}\setminus X^{\Delta}_d,\Z)$
and
$2nd+d-s>\dim (X^{\Delta}_{d+1}\setminus X^{\Delta}_d)$
$\Leftrightarrow$
$s\leq (2\rmin -2)d-1$,
we see that
$\E^1_{d+1,s}=0$ for any $s\leq (2\rmin -2)d-1$.
\end{proof}

\begin{lemma}\label{lmm: E2}
If $1\leq k\leq d$, $\theta^1_{k,s}:\E^{1}_{k,s}\to \Ed^{1}_{k,s}$ is
an isomorphism for any $s$.
\end{lemma}
\begin{proof}
Since $(X^{\Delta}_k,Y^{\Delta}_k)=(\SZ_k,\SZd_k)$ for $k\leq d$,
the assertion follows from Lemma \ref{lmm: E1}.
\end{proof}

\section{Homological stability.}\label{section: Proofs}

In this section we prove Theorem \ref{thm: unstable}, 
Theorem \ref{thm: CP} and Theorem \ref{thm: Segal}.

\paragraph{\bf 5.1. Proof of the homology stability.}
\ First, we give a proof of Theorem \ref{thm: unstable}.

\begin{proof}[Proof of Theorem \ref{thm: unstable}]
We write $r_{\rm min}=r_{\rm min}(I)$, and
consider the homomorphism
$\theta_{k,s}^t:\E^{t}_{k,s}\to \Ed^{t}_{k,s}$
of truncated spectral sequences given in (\ref{equ: theta2}).
We want to show that the positive integer $D(I;d)$ has the 
following property:
\begin{enumerate}
\item[$(*)$]
$\theta^{\infty}_{k,s}$
is  an isomorphism for all $(k,s)$ such that $s-k\leq D(I;d)$.
\end{enumerate}
By Lemma \ref{lmm: Ed}, 
$\E^1_{k,s}=\Ed^1_{k,s}=0$ if
$k<0$, or if $k\geq d +2$, or if $k=d +1$ with $s\leq (2\rmin -2)d-1$.
Because $(2\rmin -2)d-1-(d+1)=(2\rmin -3)d-2
=D(I;d)$,
we  see that:
\begin{enumerate}
\item[$(*)_1$]
if $k< 0$ or $k\geq d +1$,
$\theta^{\infty}_{k,s}$ is an isomorphism for all $(k,s)$ such that
$s-k\leq D(I;d)$.
\end{enumerate}
\par
Next, we assume that $0\leq k\leq d$, and investigate the condition that
$\theta^{\infty}_{k,s}$  is an isomorphism.
Note that the groups $\E^1_{k_1,s_1}$ and $\Ed^1_{k_1,s_1}$ are not known for
$(k_1,s_1)\in\mathcal{S}_1=\{(d+1,s)\in\Z^2:s\geq (2\rmin -2)d\}$.
By considering the differential
$d^1:E(d+\epsilon)^1_{k,s}\to E(d+\epsilon)^1_{k+1,s}$
for $\epsilon \in \{0,1\}$,
and applying Lemma \ref{lmm: E2}, we see that
$\theta^2_{k,s}$ is an isomorphism if
$(k,s)\notin \mathcal{S}_1 \cup \mathcal{S}_2$, where
$$
\mathcal{S}_2=:
\{(k_1,s_1)\in\Z^2:(k_1+1,s_1)\in \mathcal{S}_1\}
=\{(d,s_1)\in \Z^2:s_1\geq (2\rmin -2)d\}.
$$
A similar argument  shows that
$\theta^3_{k,s}$ is an isomorphism if
$(k,s)\notin \bigcup_{t=1}^3\mathcal{S}_t$, where
$\mathcal{S}_3=\{(k_1,s_1)\in\Z^2:(k_1+2,s_1+1)\in \mathcal{S}_1\cup
\mathcal{S}_2\}.$
Continuing in the same fashion,
considering the differentials
$d^t:\E^t_{k,s}\to \E^{t}_{k+t,s+t-1}$
and
$d^t:\Ed^t_{k,s}\to \Ed^{t}_{k+t,s+t-1},$
and applying the inductive hypothesis,
we  see that $\theta^{\infty}_{k,s}$ is an isomorphism
if $\dis (k,s)\notin \mathcal{S}:=\bigcup_{t\geq 1}\mathcal{S}_t
=\bigcup_{t\geq 1}A_t$,
where  $A_t$ denotes the set
$$
A_t:=
\left\{
\begin{array}{c|l}
 &\mbox{ There are positive integers }l_1,l_2,\cdots ,l_t
\mbox{ such that},
\\
(k_1,s_1)\in \Z^2 &\  1\leq l_1<l_2<\cdots <l_t,\ 
k_1+\sum_{j=1}^tl_j=d +1,
\\
& \ s_1+\sum_{j=1}^t(l_j-1)\geq (2\rmin -2)d
\end{array}
\right\}.
$$
Note that 
if this set was empty for every $t$, then, of course, the conclusion of 
Theorem \ref{thm: unstable} would hold in all dimensions (this is known to be false in general). 
If $\dis A_t\not= \emptyset$, it is easy to see that
$
a(t)=\min \{s-k:(k,s)\in A_t\}=
(2\rmin -2)d-(d+1)+t
=D(I;d)+t+1.
$
Hence, 
$\min \{a(t):t\geq 1,A_t\not=\emptyset\}=D(I;d)+2.$
Since $\theta^{\infty}_{k,s}$ is an isomorphism
for any $(k,s)\notin \bigcup_{t\geq 1}A_t$ for each $0\leq k\leq d$,
we have the following:
\begin{enumerate}
\item[$(*)_2$]
If $0\leq k\leq d$,
$\theta^{\infty}_{k,s}$ is  an isomorphism for any $(k,s)$ such that
$s-k\leq  D(I;d)+1.$
\end{enumerate}
Then, by $(*)_1$ and $(*)_2$, we see that
$\theta^{\infty}_{k,s}:\E^{\infty}_{k,s}\stackrel{\cong}{\rightarrow}
\Ed^{\infty}_{k,s}$ is an isomorphism for any $(k,s)$
if $s-k\leq D(I;d)$.
So we see that
$s_d$ is a homology equivalence through dimension
$D(I;d)$.
\end{proof}

\begin{corollary}\label{crl: Corollary I}
The inclusion map
$i_d:\Hol_d^*(S^2,X_I)\to \Omega^2_dX_I$ is a
homology equivalence through dimension $D(I;d)$.
\end{corollary}
\begin{proof}
The assertion easily follows from
Theorem \ref{thm: stable} and Theorem \ref{thm: unstable}.
\end{proof}
\par\vspace{2mm}\par

Next, we prove Theorem \ref{thm: CP}.
Since $L(I(n))=\{{\bf 0}\}$, we  see that $C_{k;I}=C_k(\C)$ if $I=I(n)$.
To prove Theorem \ref{thm: CP}, we use the spectral sequence given in (\ref{SS})
for $\epsilon\in \{0,1\}$,
\begin{equation}\label{equ: spectral sequ3}
\{E^{t;d+\epsilon}_{k,s},d^t:E^{t;d+\epsilon}_{k,s}\to 
E^{t;d+\epsilon}_{k+t,s+t-1}\}
 \Rightarrow H_{s-k}(\Hol_{d+\epsilon}^*(S^2,\CP^{n-1}),\Z).
\end{equation}
The following  result easily follows from Remark \ref{rmk: CP} and
Lemma \ref{lemma: E11}.
\begin{lemma}\label{lmm: E3}
Let $I=I(n)$ and $\epsilon\in \{0,1\}$. 
\begin{enumerate}
\item[$\I$]
If $k<0$ or $k\geq d+1+\epsilon$,
$E^{1;d+\epsilon}_{k,s}=0$ for any $s$.
\item[$\II$]
$E^{1;d+\epsilon}_{0,0}=\Z$ and $E^{1;d+\epsilon}_{0,s}=0$ if $s\not= 0$.
\item[$\III$]
If $1\leq k\leq d+\epsilon$, there is a natural isomorphism
$$
E^{1;d+\epsilon}_{k,s}\cong \tilde{H}_c^{2nk-s}(C_k(\C),\pm \Z).
\qed
$$
\end{enumerate}
\end{lemma}

\begin{proof}[Proof of Theorem \ref{thm: CP}]

Consider the homomorphism 
$\{ \tilde{\theta}_{k,s}^t:E^{t:d}_{k,s}\to E^{t;d+1}_{k,s}\}$
of spectral sequences given in
(\ref{equ: theta1}).
By using Lemma \ref{lmm: E3}
we can easily show that 
$\tilde{\theta}^{\infty}_{k,s}:E^{\infty :d}_{k,s}\stackrel{\cong}{\longrightarrow}E^{\infty ;d+1}_{k,s}$
is an isomorphism 
for any $(k,s)$ if the condition $s-k\leq D^*(d,n)=(2n-3)(d+1)-1$
is satisfied.
Hence,
$s_d:\Hol_d^*(S^2,\CP^{n-1})\to \Hol_{d+1}^*(S^2,\CP^{n-1})$
is a homology equivalence through dimension $D^*(d,n)$.
\end{proof}

Now we can also prove Theorem \ref{thm: Segal}.
\begin{proof}[Proof of Theorem \ref{thm: Segal}]
Assume that $n\geq 3$.
It follows from Theorem \ref{thm: stable}
and Theorem \ref{thm: CP} that the inclusion 
$i_d:\Hol_d^*(S^2,\CP^{n-1})\to \Omega^2_d\CP^{n-1}\simeq \Omega^2S^{2n-1}$ is a homology equivalence
through dimension $D^*(d,n)$.
However, if $n\geq 3$, the two spaces $\Hol_d^*(S^2,\CP^{n-1})$ and
$\Omega^2S^{2n-1}$ are simply connected.
Hence, $i_d$ is a homotopy equivalence through dimension $D^*(d,n)$.
\end{proof}

\paragraph{\bf 5.2. Alternative proof.}
\ We give the another proof of
Theorem
\ref{thm: CP} by using \cite{BM} and \cite{CCMM}.
First, recall the following result.
\begin{lemma}\label{lmm: CMM}
Let $(S^{2n-3})^{[k]}$ denote the $k$-fold wedge of $S^{2n-3}$,
$(S^{2n-3})^{[k]}=S^{2n-3}\wedge \cdots
\wedge S^{2n-3}$ ($k$-times), and
$D_k(2n-3)$ the $k$-th summand of $\Omega^2S^{2n-1}$
given by
$D_k(2n-3)=F(\C,k)_+\wedge_{S_k}(S^{2n-3})^{[k]}$.
\par
Then the inclusion map
$
\hat{j}_d:\bigvee_{k=1}^d D_k(2n-3) \to \bigvee_{k=1}^{d+1}D_k(2n-3)
$
is a homology equivalence through dimension 
$D^*(d,n)=(2n-3)(d+1)-1$.
\end{lemma}
\begin{proof}
Since there is a homotopy equivalence
$D_k(2n-3)\simeq \Sigma^{(2n-4)k}D_k(1)$
(\cite[Theorem 1]{CMM}) and
$D_k(1)$ is the Thom complex of the
vector bundle $F(\C,k)\times_{S_k}\R^k\to C_k(\C)$,
there is an isomorphism $\tilde{H}_j(D_k({2n-3}),\Z)\cong 
\tilde{H}_{j-(2n-4)k}(D_k(1),\Z)\cong
\tilde{H}_{j-(2n-3)k}(C_k(\C),\pm \Z)$ for any $k$
(cf. \cite[Lemma 1, page 113]{Va}).
Hence we see that
$H_j(D_{d+1}({2n-3}),\Z)=0$ for any $1\leq j\leq D^*(d,n)$ and
$\hat{j}_d$ is a homology equivalence through dimension
$D^*(d,n)$.
\end{proof}
\begin{proof}[An alternative proof of Theorem \ref{thm: CP}]
Let $D_k:=D_k(2n-3)$ and define the map
$\iota_d^n:\Hol_d^*(S^2,\CP^{n-1})\to 
\bigvee_{k=1}^d D_k(2n-3)=\bigvee_{k=1}^d D_k$
by the map of composite
\begin{equation*}
\begin{CD}
\Hol_d^*(S^2,\CP^{n-1})@>i_d>\subset>
\Omega^2_d\CP^{n-1}\simeq
\Omega^2S^{2n-1}
@>\hat{sn}>\simeq_s>
\dis
\bigvee_{k=1}^{\infty} D_k
\stackrel{pr}{\longrightarrow}
\bigvee_{k=1}^d D_k
,
\end{CD}
\end{equation*}
where
$\hat{sn}$ and $pr$ denote the Snaith splitting map \cite{Sn} and
the natural projection, respectively.
Note that $\iota^n_{d+\epsilon}$ is a stable homotopy equivalence 
for $\epsilon \in \{0,1\}$ \cite{CCMM}.
It follows from \cite{BM} that the space 
$\Hol^*(S^2,\CP^{n-1})=\coprod_{d\geq 0}\Hol_d^*(S^2,\CP^{n-1})$
has a $C_2$-structure and that the inclusion map
$i:\Hol^*(S^2,\CP^{n-1})\to \Omega^2\CP^{n-1}$ is a $C_2$-map.
Hence,  there is a  stable homotopy commutative
diagram
$$
\begin{CD}
\Hol_d^*(S^2,\CP^{n-1}) @>s_d>> \Hol_{d+1}^*(S^2,\CP^{n-1})
\\
@V{\iota_d^n}V{\simeq_s}V @V{\iota_{d+1}^n}V{\simeq_s}V
\\
\bigvee_{k=1}^d D_k @>\hat{j}_d>\subset> \bigvee_{k=1}^{d+1}D_k
\end{CD}
$$
Then
by Lemma \ref{lmm: CMM} the map $s_d$ is a homology equivalence
through dimension $D^*(d,n)$.
\end{proof}

\section{Polyhedral products and the space $X_I$.}
\label{section: polyhedral product}
In this section, we review some basic facts concerning the topology of the space
$X_I$ which were omitted in \cite{GKY1}.

\paragraph{\bf 6.1. Polyhedral products.}
We continue to assume that
$[n]=\{0,1,2\cdots ,n-1\}$.
\begin{definition}
{\rm
Let $K$ be a simplicial complex on the index set $[n]$ and let
$(\underline{X},\underline{A})=((X_0,A_0),\cdots ,(X_{n-1},A_{n-1}))$
be an $n$-tuple of pairs of spaces such that  $A_i\subset X_i$ for each
$0\leq i\leq n-1$.
Let 
$\mathcal{Z}_K(\underline{X},\underline{A})$ 
denote {\it the polyhedral product of}  $(\underline{X},\underline{A})$ 
{\it with respect to} $K$ defined
by 
\begin{equation}\label{equ: polyhedral products}
\mathcal{Z}_K(\underline{X},\underline{A})
=
\bigcup_{\sigma\in K}(\underline{X},\underline{A})^{\sigma}
\subset \prod_{k=0}^{n-1}X_k
=X_0\times \cdots \times X_{n-1},
\end{equation}
where we set
$\dis
(\underline{X},\underline{A})^{\sigma}:=
\{(x_0,\cdots ,x_{n-1})\in \prod_{k=0}^{n-1}X_k:x_k\in A_k\mbox{ if }k\notin \sigma\}.$
If $(X_j,A_j)=(X,A)$ for each $1\leq j\leq n$, we set
$(X,A)^{\sigma}=(\underline{X},\underline{A})^{\sigma}$ and
$\mathcal{Z}_K(X,A)=\mathcal{Z}_K(\underline{X},\underline{A})$.}
\end{definition}

\begin{example}[\cite{BP},  Example 6.39]\label{exa: moment-angle}
{\rm
(i)  The space $\mathcal{Z}_K=\mathcal{Z}_K(D^2,S^1)$  is  called
{\it the moment-angle complex of }$K$.
Taking $D^2=\{w\in \C: \vert w\vert \leq 1\}$ and  $S^1=\{w\in \C: \vert w\vert =1\}$, we have
$\mathcal{Z}_K\subset U(K)=\mathcal{Z}_K(\C,\C^*).$ 
\par
(ii)
The space $DJ(K)=\mathcal{Z}_K(\CP^{\infty},*)$ is  called
{\it the Davis-Januszkiwicz space of }$K$.
}
\end{example}

\begin{lemma}\label{lmm: veeI}
$\vee^I X=\mathcal{Z}_{K(I)}(X,*)$.
\end{lemma}
\begin{proof}
First, suppose that $(x_0,\cdots ,x_{n-1})\in\mathcal{Z}_{K(I)}(X,*)$.
Then $(x_0,\cdots ,x_{n-1}) \in (X,*)^{\sigma}$
for some $\sigma\in K(I)$.
Since $\sigma\in K(I)$,
$L_{\sigma}\not\subset L_{\tau}$ for any $\tau\in I$. On the other hand,  $\sigma_1\subset \sigma_2$ if and only if
$L_{\sigma_2}\subset L_{\sigma_1}$
for $\sigma_k\in [n]$ $(k=1,2)$, hence $\tau \not\subset \sigma$ for any $\tau \in I$.
Thus, for each $\tau\in I$ there is an element $j\in \tau\setminus \sigma$.
Since $j\notin \sigma$,
$x_j=*$.
Hence, $(x_0,\cdots ,x_{n-1})\in\vee^IX$
and we conclude that $\mathcal{Z}_{K(I)}(X,*)\subset \vee^IX$.
\par
Next, suppose that $(x_0,\cdots ,x_{n-1})\in \vee^IX$ and
set $\sigma =\{i\in [n]:x_i\not=*\}$.
We shall show that $\sigma\in K(I)$.
Suppose that, on the contrary, $\sigma\notin K(I)$.
Then
$L_{\sigma}\subset \bigcup_{\tau\in I}L_{\tau}$.
Then, from the definition of coordinate subspaces we see that
there has to be some element $\tau \in I$ such that 
$L_{\sigma}\subset L_{\tau}$.
Hence,
$\tau\subset \sigma$.
Because $(x_0,\cdots ,x_{n-1})\in \vee^IX$, $x_j=*$
for some $j\in \tau\subset \sigma$.
However,  since $j\in \sigma$, $x_j\not=*$, by the definition of $\sigma$,
which is a contradiction.
Hence  $\sigma\in K(I)$.
Since 
$x_j=*$ if $j\notin \sigma$,
by the definition of $\sigma$,
we see that
$(x_0,\cdots ,x_{n-1})\in (X,*)^{\sigma}\subset \mathcal{Z}_{K(I)}(X,*)$ and hence
$\vee^IX\subset \mathcal{Z}_{K(I)}(X,*)$.
\end{proof}

\begin{corollary}\label{crl: DJ}
$DJ(K(I))=\vee^I\CP^{\infty}.$
\qed
\end{corollary}

\begin{example}\label{exa: DJ(K)}
$\I$
{\rm
If $I=J(n)$,
$\dis DJ(K(J(n)))
=\CP^{\infty}\vee
\cdots \vee \CP^{\infty}$
$(n$-times).}
\par
$\II$
{\rm If $I=I(n)$,
$\dis DJ(K(I(n)))
=F_n(\CP^{\infty}),$
where
$F_n(X)$ denotes the fat wedge
defined by
$F_n(X)=
\{(x_0,\cdots ,x_{n-1})\in X^n:
x_j=*\mbox{ for some }0\leq j\leq n-1\}.$
}
\end{example}

\begin{lemma}[\cite{BP}]\label{lmm: moment-angle}
Let $K$ be a simplicial complex on the index set $[n]$, and let
$\T^n$ denote the $n$ dimensional torus, $\T^n=(S^1)^n$.
\par
$\I$ 
There is a $\T^n$-equivariant deformation retraction
$r:U(K)\stackrel{\simeq}{\longrightarrow}\mathcal{Z}_K$.
\par $\II$
$\mathcal{Z}_K$ is the
 homotopy fibre of the inclusion map
$DJ(K)\stackrel{\subset}{\longrightarrow}(\CP^{\infty})^n$.
\end{lemma}
\begin{proof}
Both assertions follow from \cite[Theorem 8.9, Corollary 6.30]{BP}.
\end{proof}

\paragraph{\bf 6.2. The topology of the space $X_I$.}
\ 
Next, we review several basic facts concerning the topology of the space $X_I$.
\begin{lemma}[\cite{GKY1}]\label{lmm: pi2XI}
The space $X_I$ is simply-connected and
$\pi_2(X_I)=\Z$.
\end{lemma}
\begin{proof}
There is a homotopy equivalence
$Y_I=U(K(I))\simeq \mathcal{Z}_{K(I)}$ and
$\mathcal{Z}_{K(I)}$ is $2$-connected
\cite[Theorem 6.33]{BP}. Thus $Y_I$ is $2$-connected.
Now the assertion easily follows from the $\C^*$-principal bundle sequence (\ref{equ: pI}).
\end{proof}
We will next consider a fibration sequence 
(\ref{equ: Segal type fibration}), which generalizes the fibration sequence given
in \cite[page 44]{Se} and
was used in \cite{GKY1}.  But the details of its 
construction were omitted in \cite{GKY1}, and instead the reader was referred to an analogous argument in \cite{Se}. 
For the sake of completeness of this paper, we give a different proof by using  basic properties of polyhedral products.

\begin{proposition}[\cite{GKY1}]\label{lmm: fibration}
There is a fibration sequence up to homotopy equivalence
\begin{equation}\label{equ: Segal type fibration}
X_I \stackrel{q_I}{\longrightarrow}\vee^I \CP^{\infty}
\stackrel{}{\longrightarrow} (\CP^{\infty})^{n-1}.
\end{equation}
\end{proposition}

\begin{proof}
Let  $G=(\C^*)^n=\T^n_{\C}$ and  consider the natural $G$-action on  $Y_I$
by coordinate-wise multiplication.
This action naturally induces a fibration sequence
\begin{equation}\label{equ: fib1}
G=\T_{\C}^n
\stackrel{\iota_1^{\p}}{\longrightarrow} Y_I \stackrel{\iota_1}{\longrightarrow} EG\times_G Y_I \stackrel{q_1}{\longrightarrow} BG=(\CP^{\infty})^n,
\end{equation}
where the map $\iota^{\p}_1$ is the inclusion map.
Similarly, let $G_1=\T_{\C}^{n-1}$ and 
recall
the $G_1$-action on
$X_I$ given by (\ref{equ: T-action}).
This also induces a fibration sequence 
\begin{equation}\label{equ: fib2}
G_1=\T_{\C}^{n-1}
\stackrel{\iota_2^{\p}}{\longrightarrow}X_I \stackrel{\iota_2}{\longrightarrow} EG_1\times_{G_1} X_I \stackrel{q_2}{\longrightarrow} BG_1=(\CP^{\infty})^{n-1},
\end{equation}
where the map
$\iota_2^{\p}$ is given by
$\iota_2^{\p}(t_1,\cdots ,t_{n-1})=[1:t_1:\cdots :t_{n-1}]$.
If $\pi_1:G \to G_1$ denotes the map given by
$
\pi_1(t_0,\cdots ,t_{n-1})=(t_0^{-1}t_1,\cdots ,t_0^{-1}t_{n-1}),
$ 
we have a commutative diagram of principal bundles
\begin{equation}\label{equ: action-diagram}
\begin{CD}
\C^*\times \C^* @>>> G\times Y_I @>{\pi_1\times p_I}>> G_1\times X_I
\\
@V{a^{\p\p}}VV @V{a}VV @V{a^{\p}}VV
\\
\C^* @>>>Y_I @>p_I>> X_I
\end{CD}
\end{equation}
where 
three vertical maps $a$, $a^{\p}$ and $a^{\p\p}$ are given by the group actions. 
Now let  $f:EG\times_GY_I\to EG_1\times_{G_1}X_I$ denote the map defined by
$f=E\pi_1\times_{\pi_1}p_I$ and let $F$ be its homotopy fibre.
Then
it follows from the diagram (\ref{equ: action-diagram}) and \cite[(2.1)]{CMN}  
that there is a homotopy commutative diagram
\begin{equation*}\label{equ: diagram of fibration}
\begin{CD}
\C^* @>g>>\C^* @>>> F @>>> B\C^*=\CP^{\infty}
\\
@VVV @VVV @VVV @VVV
\\
G @>\iota_1^{\p}>>Y_I @>\iota_1>> EG\times_GY_I @>q_1>> BG=(\CP^{\infty})^n
\\
@V{\pi_1}VV @V{p_I}VV @V{f}VV @V{B\pi_1}VV
\\
G_1 @>\iota_2^{\p}>> X_I @>\iota_2>> EG_1\times_{G_1}X_I @>q_2>> BG_1=(\CP^{\infty})^{n-1}
\end{CD}
\end{equation*}
where  all the horizontal and vertical sequences are fibration
sequences.
From the definitions of the maps $\iota_1^{\p}$ and $\iota_2^{\p}$,
we  see that the map $g$ is the identity (up to homotopy).
Hence, $F$ is contractible and $f$ is a homotopy equivalence.
Thus there is a fibration sequence (up to homotopy)
\begin{equation}\label{equ: fib4}
X_I
\stackrel{}{\longrightarrow}
EG\times_GY_I
\stackrel{}{\longrightarrow}
BG_1=(\CP^{\infty})^{n-1}.
\end{equation}
\par
On the other hand,
it follows from the proof of 
\cite[Corollary 6.29]{BP} that
there is a homotopy equivalence
\begin{equation}\label{equ: homotopy equiv}
E\T^{n}\times_{\T^{n}}\mathcal{Z}_{K(I)}
\simeq
DJ(K(I)).
\end{equation}
Hence, by Lemma \ref{lmm: KI}, Lemma \ref{lmm: moment-angle},
(\ref{equ: homotopy equiv}) and
Corollary \ref{crl: DJ},
we have homotopy equivalences
\begin{eqnarray*}
EG\times_GY_I&=&EG\times_GU(K(I))
\simeq
E\T^{n}\times_{\T^{n}}\mathcal{Z}_{K(I)}
\simeq 
DJ(K(I))
=
\vee^I\CP^{\infty}.
\end{eqnarray*}
Substituting into 
(\ref{equ: fib4}), we obtain the
the fibration sequence
(\ref{equ: Segal type fibration}).
\end{proof}
\begin{corollary}\label{crl: equivalence}
The map
$\Omega^2q_I:\Omega^2_0X_I
\stackrel{\simeq}{\longrightarrow}
\Omega^2(\vee^I\CP^{n-1})$
is a homotopy equivalence.
\qed
\end{corollary}

\section{The fan of ${\bf \Sigma}_I$ and the proof of the main result.}
\label{section: fan}
In this section we shall consider the fan ${\bf \Sigma}_{I}$
associated to the toric variety $X_I$ and give a proof of our main result
(Theorem \ref{thm: I}).
\paragraph{\bf 7.1. The fan $\Sigma_I$.}
First, consider the fan $\Sigma_I$ of the toric variety $X_I$.
\begin{definition}\label{dfn: Cone}
{\rm
Let
$\{\textbf{\textit{e}}_1,\cdots ,\textbf{\textit{e}}_{n-1}\}$ be the standard basis of 
$\R^{n-1}$
and set
$\textbf{\textit{e}}_0=-\sum_{k=1}^{n-1}\textbf{\textit{e}}_k$,
where 
$\textbf{\textit{e}}_1=(1,0,\cdots ,0),
\cdots ,
\ \textbf{\textit{e}}_{n-1}=(0,\cdots ,0,1)\in\R^{n-1}.$ 
\par
For a proper subset $\sigma \subsetneqq [n]$, let
${\rm Cone}_{\sigma}$ denote \textit{the rational polyhedral cone}
in $\R^{n-1}$ defined by
\begin{equation*}\label{equ: cone}
{\rm Cone}_{\sigma}=
\begin{cases}
{\rm Cone}(\textbf{\textit{e}}_{i_1},\cdots ,\textbf{\textit{e}}_{i_s})=
\Big\{\sum_{k=1}^sa_k\textbf{\textit{e}}_{i_k}:
a_k\geq 0\Big\} & \mbox{if }\sigma=\{i_1,\cdots ,i_s\},
\\
\{{\bf 0}\} & \mbox{if }\sigma =\emptyset .
\end{cases}
\end{equation*}
}
\end{definition}
As ${\bf \Sigma}_{I(n)}$ is the fan of the toric variety $\CP^{n-1}$, it is well known that
\begin{equation}\label{emu: fan Sigma (n)}
{\bf \Sigma}_{I(n)}=
\big\{
{\rm Cone}_{\sigma}
:
\sigma\subsetneqq [n]
\big\}.
\end{equation}
Since $X_I$ is a $\T^{n-1}_{\C}$-invariant subspace of $\CP^{n-1}$,
we have:
\begin{equation}\label{equ: cones}
 {\bf \Sigma}_I\subset 
 {\bf \Sigma}_{I(n)}
 =
 \big\{
{\rm Cone}_{\sigma}
:
\sigma\subsetneqq [n]
\big\}.
\end{equation}
By part (i) of Remark \ref{rmk: K(I)},
we also have:
\begin{equation}\label{equ: fanK(I)}
\big\{\mbox{\rm Cone}_{\sigma}:
\sigma\in K(I)\big\}\subset
 {\bf \Sigma}_{I(n)}
 =
 \big\{
{\rm Cone}_{\sigma}
:
\sigma\subsetneqq [n]
\big\}.
\end{equation}
Since $K(I)$ is a simplicial complex,
we easily see that the set 
$\big\{\mbox{\rm Cone}_{\sigma}:
\sigma\in K(I)\big\}$ is a fan in $\R^{n-1}$.
In fact,
we can show the following:

\begin{proposition}\label{prp: Sigma (I)}
${\bf \Sigma}_I=
\big\{\mbox{\rm Cone}_{\sigma}:
\sigma\in K(I)\big\}.$
\end{proposition}
\begin{proof}
Recall the construction of the fan of $\CP^{n-1}$ 
by means of the orbit-cone correspondence \cite[\S 3.2]{CLS}.
For each
$\textbf{\textit{u}}=(u_1,\cdots ,u_{n-1})\in \Z^{n-1}$,
consider the map
$\chi^{\textbf{\textit{u}}}:\C^*\to \CP^{n-1}$
defined by
$
\chi^{\textbf{\textit{u}}}(t)=[1:t^{u_1}:\cdots :t^{u_{n-1}}].$
Because $\CP^{n-1}$ is compact, 
$\dis \lim_{t\to 0}\chi^{\textbf{\textit{u}}}(t)$
exists in $\CP^{n-1}$ and
it is easy   
to see that
$
\dis \lim_{t\to 0}\chi^{\textbf{\textit{u}}}(t)
=[\epsilon]
=[\epsilon_0:\cdots :\epsilon_{n-1}],
$
where
$\epsilon_k\in \{0,1\}$ and 
$\epsilon =(\epsilon_0,\cdots ,\epsilon_{n-1})\not= {\bf 0}.$
For each
$\epsilon =(\epsilon_0,\cdots ,\epsilon_{n-1})\in
\{0,1\}^n\setminus \{{\bf 0}\}$, let
$\sigma (\epsilon)$ denote
the subset 
$\sigma (\epsilon)\subsetneqq [n]$
given by
$\sigma (\epsilon)=\{k\in [n]:\epsilon_k=0\}$.
\par
Now let 
$\epsilon 
\in\{0,1\}^n\setminus\{{\bf 0}\}$ be an element such that
 $\sigma (\epsilon )=\{i_1,\cdots ,i_s\}$.
An easy computation shows that 
$\dis \lim_{t\to 0}\chi^{\textbf{\textit{u}}}(t)
=[\epsilon ]$
iff
$\textbf{\textit{u}}\in\mbox{\rm int(Cone}(\textbf{\textit{e}}_{i_1},\cdots ,\textbf{\textit{e}}_{i_s}))$,
where 
$\mbox{int}(A)$ denotes the interior of a subspace $A\subset \R^{n-1}$.
Hence, 
by the orbit-cone correspondence
$$
\mbox{\rm Cone}(\textbf{\textit{e}}_{i_1},\cdots ,\textbf{\textit{e}}_{i_s})
\in  {\bf \Sigma}_I\ 
\Leftrightarrow \ 
[ \epsilon ]\in X_I
\Leftrightarrow \ 
\epsilon \notin 
\bigcup_{\sigma\notin K(I)}L_{\sigma} \ 
\Leftrightarrow \ 
\sigma (\epsilon )
\in K(I).
$$
In other words,
${\bf \Sigma}_I=\{\mbox{\rm Cone}_{\sigma}:\sigma\in K(I)\}$.%
\footnote{%
If $\epsilon =(1,\cdots ,1)$, 
$\sigma (\epsilon)=\emptyset$.
In this case, since
$\dis \lim_{t\to 0}\chi^{\textbf{\textit{u}}}(t)=[\epsilon]=[1:\cdots :1]$
iff $\textbf{\textit{u}}={\bf 0}$,
the zero dimensional cone $\{{\bf 0}\}\in {\bf \Sigma}_I$
corresponds to the element 
$\sigma ({\bf \epsilon})=\emptyset\in K(I)$.
}
\end{proof}

\begin{corollary}\label{crl: Fan1}
If $I=J(n)$, then
$
{\bf \Sigma}_{J(n)}=
\big\{\{{\bf 0}\},\ {\rm Cone}(\textbf{\textit{e}}_k):0\leq k\leq n-1\big\}.$
In particular, ${\bf \Sigma}_{J(n)}$ contains  all one dimensional cones
in ${\bf \Sigma}_{I(n)}.$
\qed
\end{corollary}
For a fan ${\bf \Sigma}$,
let ${\bf \Sigma} (1)$ denote the set consisting of all 
one dimensional cones in ${\bf \Sigma}.$

\begin{corollary}\label{crl: SigmaI}
${\bf \Sigma}_I(1)={\bf \Sigma}_{I(n)}(1)
=\{\mbox{\rm Cone}(\textbf{\textit{e}}_k):0\leq k\leq n-1\}$.
\end{corollary}
\begin{proof}
Note that
$X_{J(n)}\subset X_I\subset \CP^{n-1}$, and that
$X_{J(n)}$ and $X_I$ are $\T^{n-1}_{\C}$-invariant
subspaces of $\CP^{n-1}$.
Hence,
${\bf \Sigma}_{J(n)}\subset {\bf \Sigma}_I\subset {\bf \Sigma}_{I(n)}$, 
and so that 
${\bf \Sigma}_{J(n)}(1)\subset {\bf \Sigma}_I(1)
\subset {\bf \Sigma}_{I(n)}(1)$.
On the other hand, ${\bf \Sigma}_{I(n)}(1)={\bf \Sigma}_{J(n)}(1)
=\{\mbox{Cone}(\textbf{\textit{e}}_k):0\leq k\leq n-1\}$ 
by  Corollary \ref{crl: Fan1}.
Hence we see that
${\bf \Sigma}_I(1)={\bf \Sigma}_{I(n)}(1)=
\{\mbox{Cone}(\textbf{\textit{e}}_k):0\leq k\leq n-1\}$.
\end{proof}

\begin{remark}\label{rmk: homogeneous}
{\rm
It follows easily from  Corollary \ref{crl: SigmaI} that
(\ref{equ: XI=U(K)/C}) is a homogeneous coordinate representation on
the toric variety $X_I$
(cf. \cite[Chapter 5]{CLS}).
}\qed
\end{remark}

\begin{definition}\label{dfn: primitive}
Let ${\bf \Sigma}$ be a fan in $\R^{n-1}$ and let $\XS$ denote the toric variety
associated to ${\bf \Sigma}$.
As there are no non-constant holomorphic maps
from $\CP^1$ to $\T^{n-1}_{\C}$,  we assume that $\XS\not= \T^{n-1}_{\C}$
and let
${\bf \Sigma} (1)=\{\rho_1,\cdots ,\rho_r\}$.
\footnote{%
If $\XS =\T^{n-1}_{\C}$, ${\bf \Sigma}=
\{{\bf 0}\}$ and
${\bf \Sigma}(1)$ is the empty set $\emptyset$.
}
For each $1\leq k\leq r$, let 
$\textbf{\textit{n}}_k\in \Z^{n-1}$
denote {\it the primitive element of} $\rho_k$ (i.e. 
$\rho_k\cap \Z^{n-1}=\Z_{\geq 0}\cdot \textbf{\textit{n}}_k$.)
\par
We say that a set of primitive elements
$\{\textbf{\textit{n}}_{i_t}\}_{t=1}^s=
\{\textbf{\textit{n}}_{i_1},\cdots ,\textbf{\textit{n}}_{i_s}\}$
is {\it  primitive} if the whole set does not lie in any cone in ${\bf \Sigma}$ but
any of its proper subsets spans a cone in ${\bf \Sigma}$
(\cite[page 304]{CLS}).
Let $r_{\bf \Sigma}$ denote the positive integer
\begin{equation}\label{equ: rSigma}
r_{\bf \Sigma}=\min\{s\in \Z_{\geq 1}:
\{\textbf{\textit{n}}_{i_t}\}_{t=1}^s
\mbox{ is  primitive for some }
\{i_t\}_{t=1}^s\subset \{k\}_{k=1}^r\}.
\end{equation}
Note that 
$1\leq r_{\bf \Sigma}\leq r$ if $\XS \not= \T^{n-1}_{\C}$.
\end{definition}

\begin{example}
(i)
If ${\bf \Sigma}={\bf \Sigma}_{I(n)}$, $\XS =\CP^{n-1}$ and
 ${\bf \Sigma}_{I(n)} (1)=\{\mbox{\rm Cone}(\textbf{\textit{e}}_k):0\leq k\leq n-1\}$.
Since
the primitive generator of  
 $\mbox{\rm Cone}(\textbf{\textit{e}}_k)$ is
$\textbf{\textit{e}}_k$ and
$\{\textbf{\textit{e}}_0,\cdots ,\textbf{\textit{e}}_{n-1}\}$
 is the only primitive set,
 $r_{\bf \Sigma_{I(n)}}=n$.
 \par
 (ii) Similarly, because ${\bf \Sigma}_{J(n)}=\{{\bf 0},\ 
 \mbox{\rm Cone}(\textbf{\textit{e}}_k):0\leq k\leq n-1\}$,
 the set $\{\textbf{\textit{e}}_i,\textbf{\textit{e}}_j\}$ is primitive
 for any $0\leq i<j\leq n-1$. Hence, 
$r_{\bf \Sigma_{J(n)}}=2.$
 \qed
\end{example}

\begin{lemma}\label{lmm: rmin=rSigma}
$r_{\rm min}(I)=r_{{\bf \Sigma}_I}$.
\end{lemma}
\begin{proof}
Recall that 
(\ref{equ: XI=U(K)/C}) is a   homogeneous coordinate representation on
the toric variety $X_I$, and 
consider the closed variety $Z=\bigcup_{\sigma\notin K(I)}L_{\sigma}.$
Since
$Z=\bigcup_{\sigma\in I}L_{\sigma}$ 
(as in Lemma \ref{lmm: KI})
and $r_{\rm min}(I)=\min \{\mbox{card}(\sigma):\sigma\in I\},$
the maximum dimension of an irreducible component of $Z$ is
%
$n-r_{\rm min}(I)$.
However, it follows from \cite[Prop. 5.1.6]{CLS} that
$Z=\bigcup_{\{\textbf{\textit{n}}_{i_1},\cdots ,\textbf{\textit{n}}_{i_s}\}:
\ \mbox{\tiny primitive}}L_{\{i_1,\cdots ,i_s\}}$.
Hence, 
its maximum dimension is also equal to
$n-r_{{\bf \Sigma}_I}$, and we have the equality
$r_{\rm min}(I)=r_{{\bf \Sigma}_I}$.
\end{proof}
\begin{remark}\label{rmk: dimension}
{\rm
Since
${\bf \Sigma}_I(1)=
\{\mbox{\rm Cone}(\textbf{\textit{e}}_k)\}_{k=0}^{n-1}$
and the primitive element of the cone $\mbox{\rm Cone}(\textbf{\textit{e}}_k)$
is $\textbf{\textit{e}}_k$,
we can easily see that $d_1=\cdots =d_n=d$, where
$d_1,\cdots ,d_n$ are the degrees of polynomials representing the morphism in Cox 
coordinates.
Hence,  by Lemma \ref{lmm: rmin=rSigma},
$D_{*}(d_1,\cdots ,d_n;1)=(2r_{\rm min}(I)-3)d-2=D(I;d).$
Thus
the stability dimension given in Theorem \ref{thm: MV} and that in 
Theorem \ref{thm: I} are same.
}\qed
\end{remark}

\begin{lemma}\label{lmm: 1-connected}
The space $\Omega^2_dX_I$ is $2(\rmin (I)-2)$-connected.
\end{lemma}
\begin{proof}
By  
(\ref{equ: pI}), we see that there is a
homotopy equivalence
$\Omega^2_dX_I\simeq \Omega^2U(K(I))$.
However, by using
\cite[Lemma 5.4 (cf. Remark 5.5)]{KOY1}, we see that
$U(K(I))$ is $2(r_{{\bf \Sigma}_I}-1)$-connected.
Hence, as $r_{{\bf \Sigma}_I}=\rmin (I)$
(by Lemma \ref{lmm: rmin=rSigma}),
$\Omega^2U(K(I))$ is $2(\rmin (I)-2)$-connected
and so is $\Omega^2_dX_I$.
\end{proof}

\begin{lemma}\label{lmm: 1-connected(2)}
If $\rmin (I)\geq 3$, 
the space $\Hol_d^*(S^2,X_I)$ is simply connected.
\end{lemma}
\begin{proof}
Since $\rmin (I)\geq 3$, 
$\Omega^2_dX_I$ is at least  
$2$-connected
(by Lemma \ref{lmm: 1-connected}). 
Hence, by Theorem \ref{thm: GKY1}, if $d\geq 2$,
the inclusion $i_d$ induces the isomorphism
$$
i_{d*}:\pi_1(\Hol^*_d(S^2,X_I))
\stackrel{\cong}{\rightarrow}\pi_1(\Omega^2_dX_I)=0.
$$
Thus, $\Hol_d^*(S^2,X_I)$ is simply connected for $d\geq 2$.
It remains to prove  the case $d=1$.
Note that  $\Hol_1^*(S^2,X_I)$ may be identified with the space
$$
\big\{(\alpha_0,\cdots ,\alpha_{n-1})\in\C^n:
(\alpha_{i_1},\cdots ,\alpha_{i_s})\not=
(\alpha,\cdots ,\alpha )
\mbox{ for any }\alpha\in\C
\mbox{ if }\{i_k\}_{k=1}^s\in I
\big\}.
$$
Let $H(n)$ denote the group of pure braids on $n$ strings.
Then one can show that the group
 $\pi_1(\Hol_1^*(S^2,X_I))$ is isomorphic to the
quotient group of $H(n)$ and
that it  is also an abelian group
by using the same method as in \cite[\S Appendix]{GKY1}.
Thus there is an isomorphism 
$\pi_1(\Hol_1^*(S^2,X_I))\cong H_1(\Hol_1^*(S^2,X_I),\Z).$
On the other hand, since $\rmin (I)\geq 3$,
by Corollary \ref{crl: Corollary I}
there is an isomorphism
${i_1}_*:H_1(\Hol_1^*(S^2,X_I),\Z)\stackrel{\cong}{\rightarrow}
H_1(\Omega^2_dX_I,\Z)=0.$
Hence, $\pi_1(\Hol_1^*(S^2,X_I))=0$
and the assertion follows.
\end{proof}

\paragraph{\bf 7.2. Proof of the main result.}
Now we are ready to prove
Theorem \ref{thm: I}.

\begin{proof}[Proof of Theorem \ref{thm: I}]
Recall that
the map $i_d$ is a homology equivalence through dimension
$D(I;d)$
(by Corollary \ref{crl: Corollary I}).
Since $\rmin (I)\geq 3$,
it follows from
Lemma \ref{lmm: 1-connected} and Lemma \ref{lmm: 1-connected(2)}
that
$\Hol_d^*(S^2,X_I)$ and $\Omega^2_dX_I$ are simply connected.
Hence, the map $i_d$ is a
homotopy equivalence through dimension $D(I;d)$.
\end{proof}
\begin{corollary}\label{crl: N-connected}
The space $\Hol_d^*(S^2,X_I)$ is
$2(\rmin (I)-2)$-connected.
\end{corollary}
\begin{proof}
If $\rmin (I)=2$, this is trivial.
If $\rmin (I)\geq 3$,
the assertion follows from Theorem \ref{thm: I} and Lemma
\ref{lmm: 1-connected}.
\end{proof}

\noindent{\bf Acknowledgements. }
The authors should like to take this opportunity to thank  Masahiro Ohno 
for his many valuable  insights and suggestions concerning toric varieties.
\bibliographystyle{amsplain}


\end{document}